\documentclass[10pt,a4paper,twoside,english]{article} 
\usepackage[english,francais]{babel}
\usepackage[applemac]{inputenc}    
\usepackage[T1]{fontenc}
\usepackage[pdftex]{graphicx}
\usepackage{graphics}
\usepackage{color}
\usepackage{listings}                   
\usepackage{fancyhdr}
\usepackage{indentfirst}
\usepackage{array} 
\usepackage{newlfont}
\usepackage{microtype} 
\usepackage{amsfonts}
\usepackage{amscd}

\usepackage[fleqn,tbtags]{amsmath}
\usepackage{amssymb} 
\usepackage{mathrsfs} 
\usepackage{dsfont}     
\usepackage{mathtools}
\mathtoolsset{showonlyrefs}			
\numberwithin{equation}{section}      

\usepackage{latexsym}
\usepackage{amsthm}
\usepackage{fancyhdr}

\theoremstyle{definition}                          
\newtheorem{thm}{Theorem}[section]     
\newtheorem{prop}[thm]{Proposition}      
\newtheorem{lem}[thm]{Lemma}             
\theoremstyle{definition}               
\newtheorem{dfn}[thm]{Definition}

\newtheorem{ass}{Assumption}

\newtheorem{rem}[thm]{Remark}          

\def\R{\mathbb R}

\def\D{\mathbb D}

\def\P{\mathbb P}

\def\shf{{\cal F}}

\def\shm{{\cal M}}

\def\shy{{\cal Y}}

\def\1{\mathds{1}}

\def\lra{\longrightarrow}
\def\apt{\left(}
\def\cpt{\right)}
\def\be{\begin{equation}}
\def\ee{\end{equation}}

\author{{\sc Cristina Di Girolami}\footnote{
D\'epartement de math\'ematiques, UFR Sciences et Techniques, Le Mans Universit\'e, Avenue Olivier Messiaen, 72085, Le Mans CEDEX 9. Email: cristina.di\_girolami@univ-lemans.fr}
     {\sc,}\  
\ {\sc and}\ {\sc Francesco RUSSO} 
\footnote{\'Ecole Nationale Sup\'erieure des Techniques Avanc\'ees,
ENSTA Paris, Institut Polytechnique de Paris,
Unit\'e de Math\'ematiques appliqu\'ees,
828, Boulevard des Mar\'echaux
F-91120 Palaiseau. Email: francesco.russo@ensta-paris.fr}}

\date{February 1st 2020}

\title{About classical solutions of the path-dependent heat equation}

\begin{document}
\maketitle
\selectlanguage{english}
{\bf Abstract}
This paper investigates two existence theorems for the path-dependent
heat equation, which is the Kolmogorov equation related to the
window Brownian motion, considered as a $C([-T,0])$-valued
process.
We concentrate on two general existence results of its classical solutions
related to different classes of terminal conditions: the first one is given by a
cylindrical non necessarily smooth random variable, the second one is a
smooth generic functional.

\medskip

[\textbf{2010 Math Subject Classification}: \ ] \ {60H05;   60H30;  91G80 }
\medskip



\section{Introduction}

The path-dependent heat equation is a natural extension of the classical heat equation
to the path-dependent world. If the heat equation constitutes the Kolmogorov equation
associated with Brownian motion viewed as a real valued process, then the path-dependent heat equation
is the Kolmogorov equation related to the Wiener process as  $C([-T,0])$-valued process,
that we will denominate as {\it window Brownian motion}. One particularity
of $C([-T,0])$ is that it is a (even non-reflexive) Banach space
and for integrator processes taking values in it, it is not obvious
to define a stochastic integral.
In the recent past, many works were devoted
to various types of path-dependent PDE under different perspectives
(for instance under the perspective of viscosity solutions, see e.g. \cite{rosestolato_cosso, pengBSDEs, cosso_russo15b} ...),
using generally approaches close to the functional It\^o calculus of \cite{dupire}.
A  recent contribution in the study of the
path-dependent heat equation (in the spirit of Banach space) was carried on by \cite{zanco},
which considered (not necessarily smooth in time) mild type solutions,
involving at the same time a  path-dependent drift,
see also references therein for related contributions.
The problem of finding {\it classical} or  {\it smooth}
solutions has been neglected, especially using the
Banach space approach, except for some particular
final conditions, see e.g. \cite{cosso_russo15a} and
of course \cite{DGR}.

In this paper we focus on classical solutions of the path-dependent
heat equation
 with two types of terminal conditions.
In reality this work updates \cite{DGR, DGRnote}, somehow a pioneering (never published) work
of the authors, which formulated similar results in a Hilbert framework.

Let $H:C([-T,0])\longrightarrow \R$ be continuous and let $ \sigma$ be a
real constant. Even though some results can be extended to a more general context we have preferred for clarity  to work with $\sigma$ being a constant. See Section \ref{sec BrowStoF} for some results with general $\sigma:[0,T]\times \R\longrightarrow \R$.

Our  path dependent heat equation can be expressed as
\be	\label{kolmogorov}
\left\{
\begin{array}{l}
\partial_{t}u(t,\eta)+\int_{]-t,0]} D^{\perp}_{dx}u(t,\eta)\,d^{-}\eta(x) +\frac{1}{2}\sigma^2\langle D^2u(t,\eta),
\1_{\{0\}} \otimes  \1_{\{0\}}   \rangle =0  \mbox{ for }(t, \eta)\in [0,T[\times C([-T,0])\\
u(T,\eta)=H(\eta) \hspace{8cm} \mbox{ for } \eta\in C([-T,0]).\\
\end{array}
\right.
\ee
A function  $u:[0,T]\times C([-T,0])\longrightarrow \mathbb{R}$ will be a classical  solution of \eqref{kolmogorov} if it belongs to \\
$ C^{1,2}\left([0,T[\times C([-T,0])\right)\cap C^{0}\left([0,T]\times C([-T,0])\right)$ in the Fr\'echet sense and if it  verifies \eqref{kolmogorov}.
 For any given $(t,\eta)\in [0,T]\times C([-T,0])$,
  $D F \,(t,\eta) $ denotes the first order Fr\'echet derivative with respect to $\eta$,
$D^{\delta_{0}}F \,(t,\eta)$ the component of $D F \,(t,\eta)$ concentrated on the Dirac zero defined by $D^{\delta_{0}}F \,(t,\eta) := DF \,(t,\eta)(\{0\})$ and  $D^{\perp}F\, (t,\eta)$ denotes the component of $D F \,(t,\eta)$ \emph{singular} to the Dirac zero component, i.e. the measure defined by $ D^{\perp}F\, (t,\eta):=DF\, (t,\eta)-DF\, (t,\eta)\, (\{ 0\})\delta_{0}$. For every $\eta \in C([-\tau,0])$, we observe that $t \mapsto D^{\delta_{0}}F(t,\eta)$ is a real valued function.
If for each $(t,\eta)$, $D^{\perp}F\, (t,\eta)$ is absolutely continuous with respect to Lebesgue measure on the reals,
 $D^{ac}F\, (t,\eta)$ denotes its density and in particular it holds that $D^\perp_{dx} F\, (t,\eta)=D^{ac}_x F\, (t,\eta) \, dx$.

A central object appearing in the path-dependent heat equation
 PDE \eqref{kolmogorov} is the deterministic integrals denoted by
\be		\label{eqRFG}
\int_{]-t,0]}D^\perp_{dx} u(t, \eta)d^{-}\eta(x),
\ee
where $D^\perp u(t, \eta)$ is a measure on $[-T,0]$ and $\eta \in C([-T,0])$.
We will give a sense, for $-T\leq a\leq b\leq 0$, to the term
$\int_{]a,b]}D^\perp u(t, x)d^{-}\eta(x)$ as the {\it deterministic forward integral} $\lim_{\epsilon\rightarrow 0} \int_{]a,b]}D^\perp _{dx}u(t, x)\frac{\eta(x+\epsilon)-\eta(x)}{\epsilon}dx$,
see Definition \ref{def FRW and BAWKW det int}.
More generally, let $\mu$ be a  finite Borel measure on $[-T,0]$ and $f$ a c\`adl\`ag function,
we will give a sense to the integral
$
\int_{]a,b]} \mu(dx)d^-f(x) 
$. Whenever $f$ has bounded variation and $\mu$ is absolutely continuous with respect to the Lebesgue measure
 it will coincide with the classical Riemann-Stieltjes integral,
 see Proposition \ref{prop A}.


As we mentioned, we state two  existence theorems  of the
 classical solution of \eqref{kolmogorov} under two
 different types of  terminal condition given by a function $H$. In Proposition \ref{pr VBN}, we
consider as terminal condition a possibly not smooth
 function  $H$ of a finite numbers of integrals of the type $\int_{-T}^0 \varphi d^-\eta$.
The reason of validity of that result (when $\sigma \neq 0$) can be understood through the non-degeneracy feature
of Brownian motion.

In Theorem \ref{thm 2 derivate u} we suppose
 the terminal condition function $H$  to be $C^{3}(C([-T,0]))$. This result generalizes an existence results already established in
the unpublished monograph \cite{DGR} Sections 9.8 and 9.9, where we assumed a Fr\'echet  smooth dependence with respect to
$L^2([-T,0])$.


In this paper we have only concentrated our efforts on the problem of existence of a solution of \eqref{kolmogorov},
the uniqueness constituting a simpler task which can be obtained as an application of a Banach space valued It\^o formula established in \cite{DGR1}.

Let $W=(W_t)_{0\leq t\leq T}$ be a classical real Brownian motion on some probability space
$(\Omega,\mathcal{F},\P)$; $(\mathcal{F}_t)$ will denote its canonical filtration.
 $\left(W_t(\cdot)\right)$ (or simply $W(\cdot)$) stands for the window Brownian process with values in $C([-T,0])$ defined by $W_t(x) :=W_{t+x}$,
see Definition \ref{dfn WPRO}.
%

An application of our two existence results consists in obtaining a Clark-Ocone type formula
for a path-dependent random variable  $h:=H(X_T(\cdot))$,
where $X$ is a finite quadratic variation process with quadratic variation given by $[X]_t = \sigma^2 t$,
but $X$ non necessarily a semimartingale. A possible example of such process is given by $X=W+B^H$, i.e. a Brownian motion plus a fractional Brownian motion of parameter $H>1/2$ or the weak $k$-order Brownian motion of \cite{follWuYor}.

Let $u$ be the solution of \eqref{kolmogorov} provided by Proposition \ref{pr VBN} or Theorem \ref{thm 2 derivate u}. By It\^o formula, see e.g. Theorem 5.2 in \cite{DGR1},
if $u$ verifies some more technical conditions then
\be	\label{iIntoto}
h= u(0,X_0(\cdot))+\int_{0}^{T} \mathcal{L} u (t, X_t(\cdot)) \ dt + \int_{0}^{T} D^{\delta_0} u(t,X_t(\cdot))d^- X_t,
\ee
where $\mathcal{L} $ denotes differential operator for $u\in C^{1,2}\left([0,T[\times C([-T,0])\right)$ defined by
\be
\mathcal{L} u (t, \eta):=\partial_{t}u(t,\eta)+\int_{]-t,0]} D^{\perp}_{dx}u(t,\eta)\,d^{-}\eta(x)
 +\frac{1}{2}\sigma^2\langle D^2u(t,\eta),
\1_{\{0\}} \otimes  \1_{\{0\}}\rangle,
\ee
 for $(t,\eta)\in [0,T]\times C([-T,0])$. Now
by \eqref{iIntoto}
 \be 	\label{iIntoto1}
 h= u(0,X_0(\cdot))+\int_{0}^{T} D^{\delta_0} u(t,X_t(\cdot))d^- X_t,
 \ee
 where we remind that $\int_0^t Y d^-X$  is the forward integral via regularization  defined first in \cite{rv1} and \cite{rv} for $X$ (respectively $Y$) a continuous (resp. locally integrable) real process, see also \cite{Rus05}
for a survey.
 Whenever $X=W$, the forward real valued integral equals the classical It\^o integral, see Proposition 1.1 in \cite{rv1}.
 In particular, if $h\in \mathbb{D}^{1,2}$, it holds that the representation stated in \eqref{iIntoto1} coincides with the classical
 Clark-Ocone formula $
h=\mathbb{E}[h] +\int_{0}^{T}\mathbb{E} \left[ D^{m}_{t}h |\shf_{t} \right] dW_{t}		\; ,
$ i.e.
 $u(0,W_0(\cdot))=\mathbb{E}[h]$ and $ D^{\delta_0} u(t,W_t(\cdot)) = D^m_t (h | \mathcal{F}_t)$,
$D^m$ denoting the Malliavin derivative.
This follows by the uniqueness of decomposition of square integrable random variables with respect to
the Brownian filtration.
We remark that our representation \eqref{iIntoto1} can be proved in some
 cases, where $h\notin \D^{1,2}$, see e.g. Section \ref{Toy model}.



The paper is organized as follows.  After this introduction,  in Section \ref{SecPrel} we recall some preliminaries: basic notions of calculus via regularization
in finite and infinite dimension, Fr\'echet derivatives of functionals and the important subsection \ref{sec:detCalcVieRegul} about deterministic calculus via regularization.
In Section \ref{Toy model} we show the existence of a classical solution of the Kolmogorov PDE for a cylindrical $H$. Finally in Section \ref{sec 5TM} we show that existence for
$H$ being general but smooth.

\section{Preliminaries}			\label{SecPrel}

\subsection{General notations} 		\label{GN}

Let $A$ and $B$ be two general sets such that $A\subset B$;
 $\1_{A}: B\rightarrow\{ 0,1\}$ will denote the indicator function of
 the set $A$, so $\1_{A}(x)=1$ if $x\in A$ and $\1_{A}(x)=0$ if $x\notin A$.
Let $k\in \mathbb{N}\cup \{+\infty\}$,
 $C^{k}(\mathbb{R}^{n})$ indicates the set of all functions
$g:\mathbb{R}^{n}\rightarrow \mathbb{R}$ which admits all partial derivatives of order $0\leq p \leq k$ and are continuous.
If $I$ is a real interval and $g$ is a function from $I \times \mathbb{R}^{n}$ to $\mathbb{R}$ which belongs to $C^{1,2}(I \times \mathbb{R}^{n})$,
 the symbols $\partial_{t}g(t,x)$, $\partial_{i}g(t,x)$ and $\partial^{2}_{ij}g(t,x)$ will denote respectively the partial
 derivative with respect to variable $I$, the partial derivative with respect to the $i$-th
 component and the second order mixed derivative with respect to $j$-th and $i$-th component evaluated in $(t,x)\in I\times \R^{n}$.

Let $a < b$ be two real numbers, $C([a,b])$ will denote the Banach linear
 space of real continuous functions equipped with the uniform norm denoted by $\|\cdot\|_{\infty}$. Let $B$ be a
 Banach space over the scalar field $\mathbb{R}$.  The space of bounded linear mappings from $B$ to $E$ will be denoted by $L(B;E)$ and we will write
$L(B)$ instead of $L(B;B)$.
The topological dual space of $B$, i.e. $L(B; \R)$, will be denoted by $B^{\ast}$. If $\phi$ is a linear functional on $B$, we shall denote the value of $\phi$ at an element $b\in B$ either by $\phi(b)$ or $\langle \phi, b \rangle$ or even
$\prescript{}{B^{\ast}}{\langle} \phi, b\rangle_{B}$.
 Let $K$ be a compact space, $\mathcal{M}(K)$ will denote the dual space $C(K)^{\ast}$, i.e. the so-called set of all real-valued finite signed measures on $K$.
In the article, if not specified, the mention  {\it absolutely continuous}
for a real-valued measure will always
refer  to the Lebesgue measure .

Let $E,F,G$ be Banach spaces;
 we shall denote the space of $G$-valued bounded bilinear forms on the product $E\times F$ by $\mathcal{B}(E\times F;G)$ with the norm given by $\|\phi\|_{\mathcal{B}}=\sup\{ \| \phi(e,f)  \|_{G}\,:\, \|e\|_{E}\leq 1; \|f\|_{F} \leq 1 \}$. If $G=\R$ we simply denote it by $\mathcal{B}(E\times F)$. We recall that $\mathcal{B}(B\times B)$ is identified with $(B\hat{\otimes}_\pi B)^\ast $, see \cite{rr,nevpg} for more details.

We recall some notions about differential calculus in Banach spaces;
for more details the reader can refer to \cite{cartan}.
Let $B$ be a Banach space. A function $F:[0,T] \times B \longrightarrow
\mathbb{R} $, is said to be $C^{1,2}([0,T]\times B)$ (Fr\'echet),
or $C^{1,2}$ (Fr\'echet), if the following three properties are fulfilled.
1. $F$ is once continuously differentiable; the partial derivative
 with  respect to $t$ will be denoted by $\partial_{t} F :[0,T]\times B
\longrightarrow \mathbb{R}$;
2. for any $t \in [0,T]$,  $x \mapsto DF(t,x)$ is of class $C^1$
where $DF:[0,T]\times B \longrightarrow B^*$ denotes the derivative with respect to the second argument
and 3. the second order derivative with respect to the
second argument $D^2F: [0,T] \times B \rightarrow \mathcal{B}(B\times B) $ is  continuous.

If $B=C([-T,0])$, we remark that $DF$ defined on $[0,T]\times B$ takes values in $B^* \cong \mathcal{M}([-T,0])$.
For all $(t,\eta)\in [0,T]\times C([-T,0])$, we will denote by $D_{dx} F(t,\eta )$ the measure such that $\forall \; h\in C([-T,0])$
\begin{equation}   \label{eq duality deriv prima}
\prescript{}{\mathcal{M}([-T,0])}{\langle} DF(t,\eta), h \rangle_{C([-T,0])}=DF(t,\eta)(h)=\int_{[-T,0]} h(x)D_{dx} F(t,\eta ) \ .
\end{equation}
 Whenever $B=E=F=C([-T,0])$, then the space of finite signed
Borel measures on $[-T,0]^2$ is included in the space $\mathcal{B}(B\times B)$ in the following way:
  \begin{equation}\label{eqai}	
\prescript{}{\mathcal{M}([-T,0]^{2})}{\langle} \mu, \eta\rangle_{C([-T,0]^{2})} =\int_{[-T,0]^{2}}\eta(x,y)\mu(dx,dy)=\int_{[-T,0]^{2}}\eta_{1}(x)\eta_{2}(y)\mu(dx,dy)  \; .
\end{equation}

We convene that the continuous functions (and real processes) defined
on $[0,T]$ or $[-T,0]$,
 are extended by continuity to the real line.
\begin{dfn}		\label{dfn WPRO}
Given  a real continuous process $X=(X_{t})_{t\in [0,T]}$, we will call \textbf{window process}, and denote by $X(\cdot)$,
the $C([-T,0])$-valued process
\begin{equation*}
X(\cdot)=\big(X_{t}(\cdot)\big)_{t\in [0,T]}=\{X_{t}(x):=X_{t+x};
x\in [-T,0], t\in [0,T]\} \; .
\end{equation*}
\end{dfn}
$X(\cdot)$ will be  understood, sometimes
without explicit mention, as $C([-T,0])$-valued.
In this paper $B$ will be often taken to be $C([-T,0])$.

We recall now the integration by parts in Wiener space. Let $\delta$ be the {\it Skorohod integral}  or the adjoint operator of Malliavin derivative $D^{m}$ as defined in
Definition 1.3.1 in \cite{nualartSEd}.
If $u$ belongs to $Dom \; \delta$, then $\delta(u)$ is an element of $L^{2}(\Omega)$ characterized, for any $F\in \D^{1,2}$, by
\begin{equation}		\label{eq INTBPWiener}
\mathbb{E}\left[ F\, \delta(u) \right]=\mathbb{E}\left[ \int_0^{T} D^{m}_{t} F\;  u_{t} \; dt\right].
\end{equation}

\subsection{Deterministic calculus via regularization}		\label{sec:detCalcVieRegul}
Let $-T\leq a\leq b\leq 0$, we will essentially concentrate on the definite integral  on an interval  $J = ]a,b]$ and $\bar{J}=[a,b]$,where  $ a < b$ are two real numbers.
Typically, in our applications we will consider $a = -T$ or $a = - t$ and $b= 0$.  That integral will be a real number.

We start with a convention. If $f: [a,b] \rightarrow \R$ is a c\`adl\`ag function,
we extend it naturally to two possible c\`adl\`ag functions $f_J$ and $f_{\bar J} $ in real line setting
\begin{equation*}
f_{J} (x) = \left \{
\begin{array}{ccc}
f(b) &:& x > b ,  \\
 f(x)  &:& x  \in [a,b],   \\
 f(a)  &:& x  < a.
\end{array}
\right.
\quad \quad
\mbox{and}
\quad \quad
f_{\bar J} (x) = \left \{
\begin{array}{ccc}
f(b) &:& x > b , \\
 f(x)  &:& x  \in [a,b]  , \\
 0  &:& x  < a.
\end{array}
\right.
\end{equation*}

\begin{dfn}		\label{def FRW and BAWKW det int}
Let $\mu$ be a finite Borel measure on $[0,T]$, $\mu\in \mathcal{M}([-T,0])$ and $f: [a,b] \rightarrow \R$ be a c\`adl\`ag function.
We define the {\bf deterministic forward integral} on $J=]a,b]$ and on $\bar J= [a,b]$ denoted by
\begin{equation}		\label{eqFVB}
\int_{]a,b]} \mu(dx) d^-f(x) \ \Big( \mbox{or simply} \int_{]a,b]} \mu d^-f \Big) \mbox{ and}
  \int_{[a,b]} \mu(dx) d^-f(x) \ \Big( \mbox{or simply} \int_{[a,b]} \mu d^-f   \Big)
\end{equation}
as the limit of
\begin{equation}
I^-(]a,b],f,\epsilon)= \int_{]a,b]}  \frac{f_{J}(x+\epsilon) - f_{J}(x)}{\epsilon}
  \mu(dx)
  \end{equation}
 and of
 \begin{equation}
I^-([a,b],f,\epsilon)= \int_{[a,b]}  \frac{f_{\bar J}(x+\epsilon) - f_{\bar J}(x)}{\epsilon}
  \mu(dx),
  \end{equation}
when $\epsilon \downarrow 0$, provided it exists.\\

If $\mu$ is absolutely continuous
 we denote by $\mu^{ac}$ the
density with respect to Lebesgue measure.
In this case we set
\be
 \label{DetIAc}
  \int_{]a,b]} \mu d^-f :=  \int_{]a,b]} \mu^{ac} d^-f, \quad
 \int_{[a,b]} \mu d^-f :=  \int_{[a,b]} \mu^{ac} d^-f.
\ee

\end{dfn}

The first  integral on $]a,b]$ appears in the path-dependent PDE \eqref{kolmogorov}, the second one on the closed interval $[a,b]$ is fundamental in Section \ref{Toy model}.
The proposition below discusses the case when $f$ or
$\mu$ is absolutely
continuous.


\begin{prop}	\label{prop A}
Let $\mu(dx)=\mu^{ac}(x)dx$, i.e. $\mu$ be absolutely continuous with Radon-Nikodym derivative density denoted by $\mu^{ac}$. By default, the bounded variation functions will be considered as c\`adl\`ag.
\begin{enumerate}
\item If $f$ has bounded variation then
\be	\label{eq454}
\int_{]a,b]}\mu^{ac}(x) d^-f(x)= \int_{]a,b]} \mu^{ac}(x-) df (x) \ \ \mbox{(classical Lebesgue-Stieltjes integral).}
\ee
In particular, whenever $\mu^{ac} \equiv 1$,  $\int_{]a,b]} \mu^{ac}(x) d^-f(x) = f(b) - f(a). $\\
\item If the function $\mu^{ac}:[a,b]\longrightarrow \R$ is c\`adl\`ag with  bounded variation, then
\begin{enumerate}
\item
\be	
\label{eq D22}
\int_{[a,b]}\mu^{ac}(x)d^-f (x)=\mu^{ac}(b)f(b)-\int_{]a,b]}f(x)d\mu^{ac}(x),
\ee 
\item
\be	\label{eq BACDI2}
\int_{]a,b]}\mu^{ac}(x) d^- f(x)=\mu^{ac}(b)f(b)-\mu^{ac}(a)f(a)-\int_{a}^{b}f(x)d\mu^{ac}(x).
\ee
\end{enumerate}
\end{enumerate}
\end{prop}
\begin{proof}
\
The statements follow directly from the definition. Concerning the case
when the integration interval is $[a,b]$ we remark that our
definition is compatible with Definitions 4, 18, see also Proposition
8 of \cite{cosso_russo14}. By Proposition 4 ibidem,
we get item 2.(a). The other items can be established by similar
considerations and are left to the reader.
\end{proof}

\section{The existence result for cylindrical terminal condition}
\label{Toy model}
The central object of this section is Proposition \ref{pr VBN}  which gives an
existence result of the solution of the path-dependent heat equation
\eqref{kolmogorov} when the terminal condition $H$
depends on a finite number of integrals, but it is not
necessarily smooth. As we mentioned, here the idea is to exploit the
non-degeneracy aspect of the Brownian motion in the sense
that the covariance matrix of every finite dimensional
distribution is invertible.
%
In this section the standard deviation parameter $\sigma$ will be supposed  to be different from $0$.
This in opposition to the case of Section \ref{sec 5TM} where $H$ is  Fr\'echet smooth, but not necessarily
cylindrical; there $\sigma$ is allowed even to vanish.\\

We introduce now  the functional $H$.
For all $i=1,\ldots,n$, let $\varphi_{i}:[0,T]\longrightarrow \mathbb{R}$ be $C^{2}([0,T];\mathbb{R})$.
Let $f:\mathbb{R}^{n}\rightarrow \mathbb{R}$ be measurable and with linear growth.
We consider the functional
\[
H:C([-T,0])\rightarrow \mathbb{R}
\]
defined by
\begin{equation}	\label{eq 8.75}
H(\eta)=
f\left( \int_{[-T,0]}\varphi_{1}(u+T)d^-\eta (u),\ldots,\int_{[-T,0]}\varphi_{n}(u+T)d^-\eta (u) \right)  \; .
\end{equation}
%
%
We recall  that for smooth $\varphi_i$, $i\in \{1,\ldots,n\}$,
the deterministic integral $\int_{[-T,0]}\varphi_i(u+T)d^-\eta(u)$ exists pointwise, according to Definition \ref{def FRW and BAWKW det int}.
That integral exists since, by \eqref{eq D22} in Proposition \ref{prop A},
we have
\be 		\label{eq VVV}
\int_{[-T,0]}\varphi_{i}(u+T)d^-\eta (u)=\varphi_{i}(T)\eta(0)-\int_{0}^{T}\eta(s-T)d\varphi_{i}(s).
\ee
So, replacing $\eta$ by the random path $ \sigma W_T(\cdot)$
in \eqref{eq 8.75} we get
\begin{equation}		\label{eq E677}
\begin{split}
h=H(W_{T}(\cdot))& =f\left( \sigma \int_{[-T,0]}\varphi_{1}(u+T)d^{-}W_{T}(u),\ldots,  \sigma \int_{[-T,0]}\varphi_{n}(u+T)d^{-}W_{T}(u) \right)\\
&=f \left( \sigma \int_{0}^{T}\varphi_{1}(s)d^{-}W_{s},\ldots, \sigma \int_{0}^{T}\varphi_{n}(s)d^{-}W_{s} \right)\\
&=f \left( \sigma \int_{0}^{T}\varphi_{1}(s)dW_{s},\ldots,  \sigma \int_{0}^{T}\varphi_{n}(s)dW_{s} \right).
\end{split}
\end{equation}
We stress that in the first line of \eqref{eq E677} the integrands
 are deterministic forward integrals; those integrals  exist pathwise, however in the second line of \eqref{eq E677} there appear
stochastic forward integrals. The second equality is justified
because the convergence for every realization $\omega$
implies of course the convergence in probability,
which characterizes the stochastic forward integral.
The latter equality holds because  It\^o integrals with Brownian motion
are also forward integrals,
 see Proposition 1.1 in \cite{rv1}.
On the other hand, for every $i\in \{ 1,\ldots, n\}$, since
 $\varphi_{i}$ are of class $C^2$ then Proposition \ref{prop A} and in particular \eqref{eq D22} gives
\begin{equation}  \label{ipparti}
\begin{split}
\int_{0}^{t}\varphi_{i}(s)d^{-}W_{s}&=\int_{[-t,0]}\varphi_{i}(u+t)d^{-}W_{t}(u)=
\varphi_{i}(t)W_{t}-\int_{0}^t W_{s}d\varphi_{i}(s),
\end{split}
\end{equation}
 where the first equality holds by similar reasons as
for the first equality in \eqref{eq E677}.
 The second equality holds by \eqref{eq D22}.

We formulate the following {\it non-degeneracy} assumption.
\begin{ass}		\label{ass Sigmat}
For $t\in [0,T]$, we denote by $\Sigma_{t}$ the matrix in $\mathbb{M}_{n\times n}(\R)$ defined by
$$\left( \Sigma_{t}\right)_{1\leq i,j\leq n} = \left(   \int_{t}^{T}\varphi_{i}(s)\varphi_{j}(s)ds \right)_{1\leq i,j\leq n}.$$
We suppose $
\det\left( \Sigma_{t} \right)>0$ $\forall \, t\in [0,T[$.
\end{ass}
\begin{rem}
\begin{enumerate}
\item We observe that, by continuity of function
$t\mapsto \det\left( \Sigma_{t} \right) $, there is always $\tau > 0$ such that $\det\left( \Sigma_{t} \right)\neq 0 $ for all $t\in [0,\tau[$.
\item It is not restrictive to consider $\det\left( \Sigma_{0} \right)\neq 0 $ since it is always possible to orthogonalize
$(\varphi_{i})_{i=1,\ldots,n}$ in $L^{2}([0,T])$ via a Gram-Schmidt procedure.
\end{enumerate}
\end{rem}

We remember that $W$ is a classical Wiener process equipped with its canonical filtration $(\shf_{t})$.
We set $h=H(  W_{T}(\cdot))$ and we evaluate the conditional expectation $\mathbb{E}\left[ h\vert \shf_{t} \right]$.
It gives
\be \label{E3100}
\begin{split}
\mathbb{E}[h | \mathcal{F}_{t}]
&
=\mathbb{E}\left[ f \left( \sigma \int_{0}^{T}\varphi_{i}(s)dW_{s},\ldots, \sigma \int_{0}^{T}\varphi_{n}(s)dW_{s} \right) |\mathcal{F}_{t} \right] \\
&
=\Psi \left(t, \sigma  \int_{0}^{t}\varphi_{1}(s)dW_{s},\ldots, \sigma \int_{0}^{t}\varphi_{n}(s)dW_{s} \right)
  \\
&
=\Psi \left( t, \int_{[-t,0]}\varphi_{1}(u+t)d^- \sigma W_{t}(u), \ldots,  \int_{[-t,0]}  \varphi_{n}(u+t)d^- \sigma W_{t}(u) \right)
\\
&
=\Psi \left( t, \int_{[-T,0]} \varphi_{1}(u+t)d^- \sigma W_{t}(u) , \ldots ,
 \int_{[-T,0]}  \varphi_{n}(u+t)d^- \sigma W_{t}(u)
 \right) \; ,
\end{split}
\ee
where the function $\Psi: [0,T] \times \mathbb{R}^{n}\longrightarrow \mathbb{R}$ is defined by
\begin{equation}  \label{definizione Psi}
\Psi(t, y_{1},\ldots,y_{n})=\mathbb{E}\left[ f\left( y_{1}+ \sigma \int_{t}^{T}\varphi_{1}(s)dW_{s},\ldots \ldots, y_{n} +
\sigma \int_{t}^{T}\varphi_{n}(s)dW_{s} \right) \right],
\end{equation}
for any $t \in [0,T], y_1, \ldots, y_n \in \R$.
In particular
\begin{equation}
\Psi(T,y_{1},\ldots,y_{n})= f\left( y_{1},\ldots \ldots, y_{n}\right)		\; .
\end{equation}
%
%
The  second equality in \eqref{E3100}
holds because for every $1 \le i \le n$
$$ \int_0^t \varphi_n(s) \sigma dW_s =
  \int_{-t}^{0}\varphi_{n}(u+t)d^- \sigma W_{t}(u),$$
for the same reasons as in \eqref{ipparti}.
We evaluate expression \eqref{definizione Psi} introducing the density function $p$ of the Gaussian vector
$$\left(   \int_{t}^{T}\varphi_{1}(s)dW_{s} , \ldots ,\int_{t}^{T}\varphi_{n}(s)dW_{s} \right)\ , $$ whose covariance matrix equals to $\Sigma_{t}$. The function
 $p: [0,T]\times \mathbb{R}^{n}\longrightarrow \mathbb{R}$ is characterized by
\begin{equation*}
p(t,z_{1},\ldots, z_{n})=\sqrt{\frac{1}{(2\pi)^{n} \det \left( \Sigma_{t} \right) }} \exp \left\{  -\frac{(z_{1},\ldots, z_{n}) \Sigma^{-1}_{t} (z_{1},\ldots, z_{n})^{\ast}}{2}  \right\}  \; ,
\end{equation*}
and function $\Psi$ becomes
\begin{equation}	\label{eq-def 8.77c}
\Psi(t,y_{1},\ldots,y_{n})=
\left\{
\begin{array}{l}
\int_{\mathbb{R}^{n}}f\left( \tilde{z}_{1},\ldots, \tilde{z}_{n} \right)p\left(t, \frac{\tilde{z}_{1}-y_{1}}{\sigma},
 \ldots, \frac{\tilde{z}_{n}-y_{n}}{\sigma} \right) d\tilde{z}_{1}\cdots d\tilde{z}_{n},  \textrm{ if } t\in [0,T[\\
f\left( y_{1},\ldots \ldots, y_{n}\right),
 \hspace{3cm} \textrm{if } t=T   \; .
\end{array}
\right.
\end{equation}

\begin{rem} \label{RPsi}
\begin{enumerate}
\item If $f$ is not continuous, we remark that, at time $t=T$,
 $\Psi(T,\cdot)$ is a function
which strictly depends  on the representative of $f$ and not only on
its Lebesgue a.e. representative.
So $\Psi$, as a class, does not admit  a restriction to $t = T$.
\item The function $p$ is a $C^{3,\infty}([0,T[\times \mathbb{R}^{n})$ solution of
\begin{equation}	\label{eq EDDP}
\partial_{t} p (t,z_{1},\ldots, z_{n})=-\frac{1}{2} \sum_{i,j=1}^{n}\varphi_{i}(t)\varphi_{j}(t) \partial^{2}_{ij} p(t,z_{1},\ldots, z_{n})  \; .
\end{equation}
%
%
Therefore the function $\Psi$ is $C^{1,2}([0,T[\times \mathbb{R}^{n})$ and solves
\begin{equation}    \label{EDP per Psi}
\partial_{t}\Psi (t,z_{1},\ldots, z_{n}) =-\frac{\sigma^2}{2} \sum_{i,j=1}^{n}\varphi_{i}(t)\varphi_{j}(t)\partial^{2}_{ij}
\Psi (t,z_{1},\ldots, z_{n}).
\end{equation}
\end{enumerate}
\end{rem}
%
%
%
%
We define now a function $
u: [0,T] \times C([-T,0])\longrightarrow \mathbb{R} 
$
by
\begin{equation}	\label{eq-def 8.77b}
u(t,\eta)=\Psi \left( t,\int_{[-t, 0]}\varphi_{1}(s+t)d^-\eta(s),\ldots, \int_{[-t, 0]}\varphi_{n}(s+t)d^-\eta(s)\right)  \; ,
\end{equation}
where $\Psi(t,y_{1},\ldots,y_{n})$ is defined by \eqref{eq-def 8.77c}.

%
%
%
%
By the fact that, for every $i$, the functions $\varphi_{i}$ are $C^{2}$, so in
particular with bounded variation,
similarly to \eqref{eq VVV} we can write
\be		\label{eqBVBB}
\int_{[-t,0]}\varphi_{i}(s+t)d^-\eta(s)=\eta(0)\varphi_{i}(t) -\int_{0}^{t}\eta(s-t) \dot{\varphi}_{i}(s)ds,
 \;  
\ee
see  \eqref{eq D22}.
%
\begin{rem}		\label{rem AATTCOND}
By construction we have
$
u(t, \sigma W_{t}(\cdot))=\mathbb{E}[h\vert \shf_{t}]
$
and in particular $u(0,   W_{0}(\cdot))=\mathbb{E}[h]$.
\end{rem}

We state now the main proposition of this section.
\begin{prop}	\label{pr VBN}
Let $H:C([-T,0])\longrightarrow \R$ be defined by \eqref{eq 8.75} and $u:[0,T]\times C([-T,0])\longrightarrow \mathbb{R}$ be defined by \eqref{eq-def 8.77b}.
\begin{enumerate}
\item
Function $u$ belongs to $C^{1,2}\left( [0,T[\times C([-T,0])\right)$ and it is a classical solution of \eqref{kolmogorov}.
\item If $f$ is continuous then we have moreover $u \in C^{0}\left( [0,T]\times C([-T,0])\right)$.
\end{enumerate}
\end{prop}

\begin{proof}
We will see that $D^\perp u(t,\eta)$ is absolutely continuous with density
that we will denote $x \mapsto D^{ac}_{x}$, so \eqref{kolmogorov} simplifies in
\begin{equation}		\label{eq SYSTOK}			
\begin{dcases}
\mathcal{L}u\,(t,\eta)=\partial_{t}u(t,\eta)+
\int_{]-t,0]}D^{ac}_{x}u(t,\eta)\,d^- \eta(x)+\frac{1}{2} D^{2}u\,(t,\eta)(\{0,0\})=0, & \\
u(T,\eta)=H(\eta). & \\
\end{dcases}
\end{equation}
We first evaluate the derivative $\partial_{t}u\,(t,\eta)$, for a given $(t,\eta) \in [0,T] \times C([-T,0])$:
{\footnotesize{
\be  \label{derivata prima tempo}
\begin{split}
\partial_{t}u(t,\eta) &=
\partial_{t} \Psi \left(t, \int_{[-t,0]} \varphi_1 (s+t)d^- \eta(s),
\ldots, \int_{[-t,0]}\varphi_{n}(s+t)d^- \eta(s)\right)
   \\
& +\sum_{i=1}^{n} \left(\partial_{i}\Psi \left( t, \int_{[-t,0]}\varphi_{1}(s+t)d^- \eta(s),\ldots, \int_{[-t,0]}\varphi_{n}(s+t)d^- \eta(s)\right) \cdot \right.\\
& \left.\hspace{5cm} \cdot \left(\partial_{t}   \int_{[-t,0]}\varphi_{i}(s+t)d^-\eta(s)\right)\right)\\
& = \partial_{t} \Psi \left(t, \int_{[-t,0]} \varphi_1 (s+t)d^- \eta(s),
\ldots, \int_{[-t,0]}\varphi_{n}(s+t)d^- \eta(s)\right) \\
&+\sum_{i=1}^{n} \left(\partial_{i}\Psi \left( t,
\int_{[-t,0]}\varphi_{1}(s+t)d^-\eta(s),\ldots,
\int_{[-t,0]}\varphi_{n}(s+t)d^-\eta(s)\right)\cdot I_i \right),
\end{split}
\ee
}}
where
$$ I_i := \left(\int_{]-t,0]}\dot{\varphi}_{i}(s+t)d^-\eta(s) \right).$$
Indeed, by usual theorems of  Lebesgue integration theory and by Proposition \ref{prop A}, \eqref{eq D22} and \eqref{eq BACDI2},
for every $ 1 \le i \le n$, we obtain
\begin{equation}		\label{riferimento 1}
\begin{split}
\partial_{t}  \left(  \int_{[-t,0]}\varphi_{i}(s+t)d^- \eta(s)\right) & = \partial_{t} \left( \eta(0)\varphi_{i}(t)-\int_{-t}^0 \eta(s)\dot{\varphi}_{i}(s+t)ds \right)\\
&=\eta(0)\dot{\varphi}_{i}(t)-\eta(-t)\dot{\varphi}_{i}(0^+)-  \int_{-t}^0 \eta(s)\ddot{\varphi}_{i}(s+t)ds=I_i. 
\end{split}
\end{equation}

In order to evaluate the derivatives of $u$ with respect to $\eta$,  we observe that by \eqref{eq-def 8.77b} and \eqref{eqBVBB}
we get
\be
u(t,\eta)=\Psi \left( t,\eta(0)\varphi_{1}(t) -\int_{0}^{t}\eta(s-t) \dot{\varphi}_{1}(s)ds ,\ldots,\eta(0)\varphi_{n}(t) -\int_{0}^{t}\eta(s-t) \dot{\varphi}_{n}(s)ds \right).
\ee
For every $t\in [0,T]$, $\eta\in C([-T,0])$, the first derivative $Du$ evaluated at $(t,\eta)$ is the measure on $[-T,0]$ defined by
\begin{equation}  \label{prima deriv rispetto allo spazio}
\begin{split}
D_{dx} u(t,\eta)  & = D^{ac}_{x}u(t,\eta)\,dx + D^{\delta_{0}}u(t,\eta)\delta_{0}(dx) \quad \mbox{with}\\
D^{ac}_{x}u(t,\eta)
&
=- \sum_{i=1}^{n}
\left( \partial_{i} \Psi   \left( t, \int_{[-t,0]}\varphi_{1}(s+t)d^- \eta(s),\ldots, \int_{[-t,0]}\varphi_{n}(s+t)d^- \eta(s)\right)
 \right)\cdot \\
 & \hspace{7cm} \cdot \Big(   \mathds{1}_{[-t,0]}(x) \dot{\varphi}_{i}(x+t) \Big)\ , \\
 D^{\delta_{0}}u(t,\eta)
 &
 = \sum_{i=1}^{n}
\left( \partial_{i} \Psi   \left( t, \int_{[-t,0]}\varphi_{1}(s+t)d^- \eta(s),\ldots, \int_{[-t,0]}\varphi_{n}(s+t)d^- \eta(s)\right)
 \right)\cdot   \varphi_{i}(t)			\; .
 \end{split}
\end{equation}
As anticipated, we observe that $x\mapsto D^{ac}_{x}u\,(t,\eta)$ has bounded variation.
%

Deriving again in a similar way, for every $t\in [0,T]$, $\eta\in C([-T,0])$, the second order derivative $D^{2}u$ evaluated at $(t,\eta)$ gives
\begin{equation}   \label{derivata seconda spazio}
\begin{split}
D^{2}_{dx,dy}u(t,\eta)=\sum_{i,j=1}^{n}& \left( \partial_{i,j}^{2}\Psi \left( t, \int_{[-t,0]}\varphi_{1}(s+t)d^- \eta(s),\ldots, \int_{[-t,0]}\varphi_{n}(s+t)d^- \eta(s)\right) \right)\cdot\\
&\hspace{-1cm} \cdot \Bigg( \varphi_{i}(t)\varphi_{j}(t)\delta_{0}(dx)\delta_{0}(dy)-\varphi_{i}(t) \mathds{1}_{[-t,0]}(x) \ddot{\varphi}_{j}(x+t)\delta_{0}(dy)\\
&\hspace{-1cm}  -\varphi_{j}(t) \mathds{1}_{[-t,0]}(y)\ddot{\varphi}_{i}(y+t)\delta_{0}(dx) +  \mathds{1}_{[-t,0]}(x)\mathds{1}_{[-t,0]}(y)\ddot{\varphi}_{i}(x+t)
\ddot{\varphi}_{j}(y+t)\Bigg)  \; .
\end{split}
\end{equation}
We get
{\scriptsize{
\be
 \label{DOrth}
\int_{]-t,0]}D^{ac}_{x}u(t,\eta)\,d^-\eta(x) =
  \sum_{i=1}^{n} \left( \partial_{i}\Psi \left( t, \int_{[-t,0]}\varphi_{1}(s+t)d^-\eta(s),\ldots, \int_{[-t,0]}\varphi_{n}(s+t)d^-\eta(s)\right) \right) \cdot I_{i}.
\ee
}}
Using \eqref{EDP per Psi}, \eqref{derivata prima tempo},
\eqref{DOrth} and \eqref{derivata seconda spazio} we obtain that
$
\mathcal{L}u\,(t,\eta) = 0.
$
Condition $u(T,\eta)=H(\eta)$ is trivially verified by definition. This concludes the proof of point 1.

Point 2. is immediate.
\end{proof}

\begin{rem}
In this section we have often used the concept of deterministic
forward integral on a closed interval $[-t,0]$,
 given in Definition \ref{def FRW and BAWKW det int}
\begin{equation}	\label{eq ERF4}
\int_{[-t,0]}\varphi_{i}(s+t)d^-\eta(s),		
\quad \mbox{instead of} \quad 
\int_{]-t,0]}\varphi_{i}(s+t)d^-\eta(s).		
\end{equation}
Since $W_{0}=0$,
the two integrals are the same when we replace  $\eta=W_{t}(\cdot)$ so
$$ \int_{[-t,0]}\varphi_{i}(s+t)d^-\eta(s) \vert_{\eta=W_{t}(\cdot)}=
 \int_{]-t,0]}\varphi_{i}(s+t)d^- \eta(s) \vert_{ \eta=W_{t}(\cdot)}.$$
The choice of the left expression in \eqref{eq ERF4}, which is compatible with the fact of considering \\
$\int_{]-t,0]}D^{ac}_{x}u(t,\eta)\,d^-\eta(x)$ in \eqref{kolmogorov},
 is justified since
\[
t\mapsto \int_{]-t,0]}\varphi_{i}(s+t)d^-\eta(s)
\]
is not differentiable.
 \end{rem}
 
\section{The existence result for smooth Fr\'echet terminal condition}  \label{sec 5TM}

\subsection {Preliminary considerations}

In this section, we will prove an existence theorem for classical solutions of \eqref{kolmogorov} under smooth Fr\'echet terminal condition.
In order to define explicitly the solution of the PDE, we need to introduce
two central objects for this section: the Brownian stochastic flow which is a real valued stochastic flow denoted by $\left( X_t^{s,x}\right)_{0\leq s\leq t\leq T, x\in \R}$
 and the functional Brownian stochastic flow
which is a $C([-T,0])$-valued stochastic flow denoted by $\left( Y_{t}^{s,\eta}\right)_{0\leq s\leq t\leq T, \eta \in C([-T,0])}$.
\begin{dfn}			\label{def flow}
Let $\Delta:=\{ (s,t)| \ 0\leq s\leq t\leq T\}$ and $\eta\in C([-T,0])$. We define the flows that will appear in this section.
\begin{enumerate}
\item We denote by $\left( X_t^{s,x}\right)_{0\leq s\leq t\leq T, x\in \R}$
the real-valued random field defined
 by
\begin{equation}			\label{defXF}
(s,t,x)\mapsto X_t^{s,x}=x+\sigma(W_t-W_s) \ .
\end{equation}
This will be called \textbf{Brownian stochastic flow}.
\item 
We denote by $\left( Y_{t}^{s,\eta}\right)_{0\leq s\leq t\leq T, \eta \in C([-T,0])}$
the $C([-T,0])$-valued random field defined
 by
\begin{equation}			\label{eq Br flow}
(s,t,\eta)\mapsto Y_{t}^{s,\eta}(x)=
\left\{
\begin{array}{ll}
\eta(x+t-s)  			& x\in [-T,s-t[\\
\eta(0)+\sigma \apt W_{t}(x)-W_{s}\cpt	& x\in [s-t,0].
\end{array}
\right.
\end{equation}
This will be called \textbf{functional Brownian stochastic flow}.
\end{enumerate}
\end{dfn}

 Let $H:C([-T,0])\longrightarrow \R$ be the functional appearing in \eqref{kolmogorov} and a path-dependent random variable  $h:=H(\sigma W_{T}(\cdot))$.
 We define the functional  $u:[0,T]\times C([-T,0])\longrightarrow \R$ by
\begin{equation}		\label{eq u}
u(t,\eta)=\mathbb{E}\left[ H\big( Y_{T}^{t,\eta} \big)\right].
\end{equation}
%
Since $\sigma W_{T}(\cdot)=Y_{T}^{t,\sigma W_{t}(\cdot)}$ we have
\begin{equation} \label{Definitionu}
\mathbb{E}\left(h \vert \shf_{t}  \right) =
\mathbb{E}\left[H\big (\sigma W_{T}(\cdot) \big)
\vert \shf_{t}  \right] =
\mathbb{E}\left[ H\big( Y_{T}^{t,\sigma W_{t}(\cdot)}\big) \vert \shf_{t}\right]=
u(t,\sigma W_{t}(\cdot)).
\end{equation}
For this reason, $u$ defined in \eqref{eq u} is a natural candidate to be a solution of \eqref{kolmogorov}.
In Theorem \ref{thm 2 derivate u} we will show, under smooth regularity of $H$, that such an $u$ is sufficiently smooth
to be a classical solution of the
path-dependent heat equation \eqref{kolmogorov}.

We dedicate next two subsections to investigate some properties of $Y^{t,\eta}_T$ that we will use in the proof of the main theorem. The  Section
\ref{sec Flow} below contains the
general results for the flows introduced in Definition \ref{def flow}.
In Section \ref{sec BrowStoF} we will introduce the Markovian stochastic flow for a general $\sigma:[0,T]\times \R\longrightarrow \R$, which coincides with the Brownian stochastic flow when $\sigma$ is constant.
We will derive some properties for this flow that we need in the theorem. We recall that, given $X$ and $Y$ two random elements taking values in the same space, we write $X\sim Y$ if they have the same law.
From now on a realization $\omega \in \Omega$ will be often fixed.

\subsection{Some properties of the Brownian (resp. functional
Brownian) flow} \label{sec Flow}
First of all we observe that the functional Brownian stochastic flow is time-homogeneous in law.
\begin{prop}			\label{RT55}
$Y^{s,\eta}_{t}$ and $Y^{0,\eta}_{t-s}$ have the same law as $C([-T,0])$-valued random variables. In particular, for every $x\in [-T,0]$, $Y^{0,\eta}_{t-s}(x) \sim Y^{s,\eta}_{t}(x)$.
\end{prop}
\begin{proof}
It follows from the two following arguments.
For $x\in [-T,s-t]$, $Y^{s,\eta}_{t}(x)$ and $Y^{0,\eta}_{t-s}(x)$
deterministic and are equal to $\eta(x+t-s)$.
 For $x\in [s-t, 0]$,
the real-valued processes
 $Y^{s,\eta}_{t}(x)=\eta(0)+\sigma \apt W_{t}(x)-W_{s}\cpt $ and $Y^{0,\eta}_{t-s}(x)=\eta(0)+\sigma \apt W_{t-s}(x)-W_{0}\cpt$, have the same law by well-known properties of Brownian motion.
\end{proof}

The next proposition concerns the continuity of the
field $Y^{s,\eta}_t$
 with respect to its three variables.
\begin{prop} 					\label{Pdef flow}
$(Y_{t}^{s,\eta})_{0\leq s\leq t\leq T , \,\eta\in C([-T,0])}$ is a continuous random field.
\end{prop}

\begin{proof} \
As usual in this section $\omega \in \Omega$ is fixed
and  $\varpi_{\eta}$ (resp. $\varpi_{W(\omega)}$) is  respectively the
modulus of continuity of $\eta$ (resp. the Brownian path $W(\omega)$). \\
Let  $(s,t, \eta)$ be such that $ 0 \le  s \le t \le T, \eta  \in C([-T,0])$
and a sequence $(s_n,t_n, \eta_n)$ also
such  that  $ 0 \le  s_n \le t_n \le T, \eta_n  \in C([-T,0])$
with
$$ \lim_{n \rightarrow \infty} \left(\vert s - s_n\vert +
\vert t - t_n\vert + \Vert \eta  - \eta_n \Vert_\infty \right) = 0.$$
We have to show that $Y^{s_n,\eta_n}_{t_n} \longrightarrow
Y^{s,\eta}_{t}$ in $C([0,T]$, when $n \rightarrow \infty$ i.e. uniformly.
For $x \in [0,T]$, we evaluate
$$
 \vert Y^{s_n,\eta_n}_{t_n} - Y^{s,\eta}_{t} \vert(x) \le
(I_1(n) + I_2(n) + I_3(n))(x),$$
 where
\begin{eqnarray*}
I_1(n)(x) &=&  \vert Y^{s_n,\eta_n}_{t_n} - Y^{s_n,\eta}_{t_n} \vert (x)\\
I_2(n)(x)  &=&  \vert Y^{s,\eta}_{t_n} - Y^{s,\eta}_{t} \vert (x)\\
I_3(n) (x)&=&  \vert Y^{s_n,\eta}_{t_n} - Y^{s,\eta}_{t_n} \vert(x).
\end{eqnarray*}
By Definition \ref{def flow}, it is easy to see that
$
\Vert I_1(n) \Vert_\infty \le  \Vert \eta - \eta_n \Vert_\infty
+ \vert \eta_n(0) - \eta(0)\vert 
\le 2 \Vert \eta - \eta_n \Vert_\infty.
$
Since $I_3(n)$ behaves similarly to  $I_2(n)$,
we only show  that
$ \label{eq conv unif flow}
\lim_{n \rightarrow\infty} I_2(n) = 0.
$
Without restriction of generality, we will suppose that
 $t_n \le t$
for any $n$, since the  case when the sequence $(t_n)$ is  greater
or equal than $t$, could be treated  analogously.
We observe that following equality holds:
\begin{equation}		\label{eq SDFV}
\begin{split}
(Y_{t_{n}}^{s,\eta}-Y_{t}^{s,\eta})(x)&=
\eta(x+t_{n}-s)\1_{[-T,s-t_{n}]}(x) -\eta(x+t-s)  \1_{[-T,s-t]}(x)\\
&\hspace{1cm} +\left(\eta(0)+ \sigma W_{t_{n}}(x)- \sigma W_{s}\right) \1_{[s-t_{n},0]}(x)
\\
&\hspace{1cm}
-\left(\eta(0)+ \sigma W_{t}(x)- \sigma W_{s}\right) \1_{[s-t,0]}(x)=\\
&=\left( \eta(x+t_{n}-s) -\eta(x+t-s) \right) \1_{[-T,s-t]}(x)\\
&\hspace{1cm} +\left(  \eta(x+t_{n}-s)- \eta(0)- \sigma W_{t}(x)+ \sigma
W_{s}\right)
\1_{[s-t,s-t_{n}]}(x) \\
&\hspace{1cm} +\left( \sigma W_{t_{n}}(x) - \sigma W_{t}(x)\right) \1_{[s-t_{n},0]}(x)  \; .
\end{split}
\end{equation}
Using \eqref{eq SDFV} to evaluate $\Vert I_2(n) \Vert_\infty$
we obtain
\[
\begin{split}
\sup_{x\in [-T,0]}\vert Y_{t_{n}}^{s,\eta}(x)-Y_{t}^{s,\eta}(x)\vert &
\leq \sup_{x\in[-T,0]} \vert \eta(x+t_{n}-s)-\eta(x+t-s)  \vert \\
&\hspace{1cm} + \sup_{x\in[s-t,s-t_{n}]} \vert \eta(x+t_{n}-s)- \eta(0)\vert\\
&\hspace{1cm}
 +  \sup_{x\in[s-t,s-t_{n}]} \sigma \vert W_{t}(x)-W_{s} \vert \\
&\hspace{1cm} + \sup_{x\in[-T,0]} \sigma \vert W_{t_{n}}(x)-W_{t}(x) \vert  \leq \\
&
\leq 2\; \varpi_{\eta}(|t_{n}-t|)+ 2\; \sigma \varpi_{W(\omega)}(|t_{n}-t|)
\xrightarrow[n\longrightarrow +\infty]{} 0.
\end{split}
\]
Since $\eta$ and $W(\omega)$ are uniformly
continuous on the compact set $[0,T]$ both moduli of continuity
converge to zero when $t_{n}\rightarrow t_{0}$.
\end{proof}
%
%

At this point we make some simple observations that will be often used in the sequel.
\begin{rem}
\item
\begin{enumerate}
\item
There are universal constants $C_{1}$, $C_{2}$, $C_{3}$ and $C_{4}$ such that,
for every $t \in [0,T], \epsilon >0$ with $t + \epsilon \in [0,T] $, it holds
\begin{equation}			\label{eq MAGG1}
\begin{split}
\left\|
Y_{T}^{t,\eta}
\right\|_{\infty}
\leq
C_{1}\; \left( 1+\|\eta\|_{\infty}+ \sup_{t\in [0,T] }
\sigma \vert W_{t}\vert  \right)\; \\
\left\|
Y_{T}^{t+\epsilon,\eta}
\right\|_{\infty}
\leq
C_{2}\; \left( 1+\|\eta\|_{\infty}+ \sigma \sup_{t\in [0,T] }\vert W_{t}\vert  \right)
\end{split}
\end{equation}
and
\begin{equation}			\label{eq MAGG0}
\left\|
Y_{0}^{T-t,\eta}
\right\|_{\infty}
\leq
C_{3} \; \left( 1+\|\eta\|_{\infty}+ \sigma \sup_{t\in [0,T] }\vert W_{t}\vert  \right)		\; .
\end{equation}
 \eqref{eq MAGG1} implies that, for any $\alpha \in [0,1],
 t \in [0,T], \epsilon $ with $t + \epsilon \in [0,T] $,
 it holds
\begin{equation}			\label{eq MAGG}
\left\|
\alpha Y_{T}^{t+\epsilon,\eta}+(1-\alpha) Y_{T}^{t+\epsilon,Y_{t+\epsilon}^{t,\eta } }
\right\|_{\infty}
\leq
C_{4}\; \left(1+\|\eta\|_{\infty}+ \sigma \sup_{t\in [0,T] }\vert W_{t}\vert  \right).
\end{equation}
\item
 For any $\alpha \in [0,1], t \in [0,T]$ it holds
\begin{equation}		\label{eq CONVC}
 \alpha Y_{T}^{t+\epsilon,\eta}+(1-\alpha) Y_{T}^{t+\epsilon,Y_{t+\epsilon}^{t,\eta } }
 \xrightarrow[ \epsilon\longrightarrow 0 ]{ C([-T,0]) } Y_{T}^{t,\eta}		\; .
\end{equation}
In fact expanding the term $Y_{T}^{t+\epsilon, Y_{t+\epsilon}^{t,\eta}}$,
 which equals $Y_{T}^{t,\eta}$, we obtain
\[
\left\|
 \alpha Y_{T}^{t+\epsilon,\eta}+(1-\alpha) Y_{T}^{t+\epsilon,Y_{t+\epsilon}^{t,\eta } }-Y_{T}^{t,\eta}
 \right\|_{\infty}=
\alpha \left\| Y_{T}^{t+\epsilon,\eta}-Y_{T}^{t,\eta} \right\|_{\infty}.
 \]
The  right-hand side  converges to zero because
of  Proposition \ref{Pdef flow}.
\item In the sequel we will make an explicit use of
the expression below:
{\scriptsize{
\begin{equation}		\label{eq DIFFER}
\left( Y_{T}^{t+\epsilon, \eta}-Y_{T}^{t,\eta}\right)(x)
=
\left\{
\begin{array}{ll}
\eta(x+T-t+\epsilon)-\eta(x+T-t)		&  x\in [-T,t-T]\\
\eta(x+T-t+\epsilon)-\eta(0)-\sigma W_{T}(x)+ \sigma W_{t}		&	x\in
[t-T,t-T+\epsilon]\\
\sigma W_{t}- \sigma W_{t+\epsilon}				& 	
		x\in [t-T+\epsilon,0].
\end{array}
\right.
\end{equation}
}}
\end{enumerate}
\end{rem}
%
%
%
%

\subsection{About a Markovian stochastic flow and functional Markovian stochastic flow
}		\label{sec BrowStoF}

%
%
The Brownian (resp. functional Brownian) stochastic flow can be generalized considering $\sigma:[0,T] \times \R \rightarrow \R$  Lipschitz with linear growth, i.e. not necessarily constant. We introduce the Markovian flow and we show some properties.

Let $\sigma, b:[0,T] \times \R \rightarrow \R$ being Lipschitz functions
with linear growth. Let, for every $s \in [0,T[, x \in \R$,
$X = X^{s,x} $ be the solution of the SDE
\begin{equation}		\label{eq SDE}
X_t = x + \int_s^t \sigma(u,X_u)dW_u + \int_s^t b(u,X_u)du, \quad t \in [s,T].
\end{equation}
Let again $\Delta:=\{ (s,t)| \ 0\leq s\leq t\leq T\}$.
It is well-known that the real-valued random field
  $(s,t,x)\mapsto X_t^{s,x}$ defined over $\Delta\times  \R \longrightarrow \R$,
admits a continuous modification.


\begin{dfn}			\label{def MArk flow}
\begin{enumerate}
\item
The random field $(s,t,x)\mapsto X_t^{s,x}$
 will be called \textbf{Markovian stochastic flow}.
\item 
We denote by $\left( Y_{t}^{s,\eta}\right)_{0\leq s\leq t\leq T, \eta \in C([-T,0])}$
the  random field defined over
$\Delta\times  C([-T,0]) \longrightarrow C([-T,0])$ by
\begin{equation}			\label{eq Br flowBis}
(s,t,\eta)\mapsto Y_{t}^{s,\eta}(x)=
\left\{
\begin{array}{ll}
\eta(x+t-s)  			& x\in [-T,s-t[\\
X_t^{s,\eta(0)}	& x\in [s-t,0].
\end{array}
\right.
\end{equation}
This will be called \textbf{functional Markovian stochastic flow}.
\end{enumerate}
\end{dfn}
\begin{rem}      \label{SBMBar}
  \begin{enumerate}
 \item The Brownian flow $(X^{s,x}_t)$ introduced in Definition \ref{def flow} is a particular case of the Markovian flow when $\sigma(t,x)=\sigma$, $\sigma$ a constant.
We could have formulated this chapter in this more general framework but for simplicity of exposition we have restricted us to the case $\sigma$ constant.
\item
The Markovian stochastic flow
 verifies the  flow property for $0\leq s\leq t\leq r\leq T$,
\begin{equation}		\label{FP}
X_{r}^{s,x}=X_r^{t,X_t^{s,x}} \ .
\end{equation}
We set
\begin{equation}		\label{eq flow}
Y_{t}^{s,\eta}(x)
=
\left\{
\begin{array}{ll}
\eta(x+t-s)  			& x\in [-T,s-t]\\
X_{t+x}^{s,\eta(0)}	& x\in [s-t,0]. \\
\end{array}
\right.
\end{equation}
The functional flow $(Y_{t}^{s,\eta})$ coincides of course with \eqref{eq Br flow}
when $(X^{s,x}_t)$ is given by \eqref{defXF}.
\end{enumerate}
\end{rem}
The following lemma shows a ``flow property'' for the functional flow.
\begin{lem}  \label{lemma 1}
Let $\eta\in C([-T,0])$, for $0\leq  s \leq t\leq r \leq T$. Then
\begin{equation}			\label{eq: flow property}
Y_{r}^{s,\eta}=Y_{r}^{t,Y_{t}^{s,\eta}} .
\end{equation}
\end{lem}
\begin{proof} \
It follows from the flow property \eqref{FP} for the Markovian stochastic flow.\\
For fixed $\omega \in \Omega$, we inject $\tilde{\eta}=Y_{s}^{t,\eta}$ into $Y_{r}^{t,\tilde{\eta}}$
obtaining
\[
Y_{r}^{t,Y_{t}^{s,\eta}}(x)=
\left\{
\begin{array}{ll}
\eta(x+r-s) 				& x\in [-T,s-r]\\
X^{s,\eta(0)}_{r+x}& x\in [s-r,t-r]\\
X^{t,\tilde{\eta}(0)}_{r+x}=X_{r+x}^{t,X_t^{s,\eta(0)}}=X^{s,\eta(0)}_{r+x}	& x\in [t-r,0]
\end{array}
\right\}
=Y_{r}^{s,\eta}(x) \ ,
\]
which concludes the proof of the Lemma.
\end{proof}

We concentrate now on the derivatives of the functional Markovian stochastic flow. Let $t \in [0,T[$.

By \eqref{eq flow} we remind that
\be   \label{TGB11}
Y_{T}^{t,\eta}(\rho)=
\left\{
\begin{array}{ll}
\eta(\rho+T-t)  			& \rho \in [-T,t-T[\\
X^{t,\eta(0)}_{T+\rho}	& \rho\in [t-T,0].
\end{array}
\right.
\ee
It is possible to calculate formally
the first
and second  derivatives of $Y_{T}^{t,\eta}(\rho)$
for $\rho \in [-T,0]$.


%
\begin{rem}		\label{rem RA}
For $\rho\in [-T,0]$ then  $Y_T^{t,\cdot}(\rho): C([-T,0])\times \Omega\lra \R$ and
$DY_T^{t,\cdot}(\rho):C([-T,0])\times \Omega\lra \apt C([-T,0]) \cpt ^*=\mathcal{M}([-T,0])$.
In particular if $f\in C([-T,0])$,
\be \label{eq DerY2bis}
\prescript{}{\mathcal{M}([-T,0])}{\langle} DY_{T}^{t,\eta}(\rho)\ , \ f \rangle_{C([-T,0])}=\int_{[-T,0]} f(x ) D_{dx} Y^{t,\eta}_{T}(\rho,\omega).
\ee

In particular
we have
\be			\label{eq DerY2}	
D_{dx} Y^{t,\eta}_{T}(\rho)
=
\left\{
\begin{array}{ll}  \delta_{\rho+T-t}(dx) & \rho\in [-T,t-T[\\
\delta_{0}(dx) \partial_{\xi} X^{t,\eta(0)}_{T+\rho}
& \rho\in [t-T,0]
\end{array}
\right.
\ee
and
\be		\label{q DerY22}
D^2_{dy\,dx} Y^{t,\eta}_{T}(\rho)
=
\left\{
\begin{array}{ll}  0 & \rho\in [-T,t-T[\\
\delta_{0}(dx)\delta_{0}(dy) \partial^2_{\xi\xi} X^{t,\eta(0)}_{T+\rho}
& \rho\in [t-T,0].
\end{array}
\right.
\ee
\end{rem}
Avoiding some technicalities it is possible to evaluate the
first and second derivatives of the functional flow itself.
In the sequel $\eta$ will always be a generic element in $C([-T,0])$.
Let $(X^{s,x}_t)$  be the real stochastic flow as in \eqref{defXF} and the associated functional stochastic flow $(Y^{s,\eta}_t)$ as in Definition \ref{def flow}.

\begin{lem}			\label{lem L1re}
Let $t \in [0,T[$.
\begin{enumerate}
\item
 The map $Y_T^{t,\cdot}:C([-T,0])\times \Omega\lra C([-T,0])$ acting as $\eta\mapsto Y_T^{t,\eta}$ is of class $C^{2}\apt C([-T,0]) \ ; C([-T,0]) \cpt$ a.s.
\item The derivatives
$DY_T^{t,\cdot}:C([-T,0])\times
\Omega\lra \mathcal{L}\apt C([-T,0]); C([-T,0])\cpt$ and $D^2 Y_T^{t,\cdot}:C([-T,0])\times
\Omega\lra \mathcal{B}\apt C([-T,0])\times C([-T,0]) ; C([-T,0])\cpt$
are characterized as follows.
 For $f, g\in C([-T, 0])$ we have
\be			\label{eq DerY}
\begin{split}
\rho\mapsto \int_{[-T,0]}D_{dx} Y^{t,\eta}_{T}(\rho)f(x)
&=
\left\{
\begin{array}{ll}  f(\rho+T-t) & \rho\in [-T,t-T[\\
f(0) \, \partial_{\xi} X^{t,\eta(0)}_{T+\rho}	& \rho\in [t-T,0]
\end{array}
\right.
\\
\mbox{ and }\\
\rho\mapsto \int_{[-T,0]^2}D^2_{dy\,dx} Y^{t,\eta}_{T}(\rho)f(x)g(y)
&=
\left\{
\begin{array}{ll}  0 & \rho\in [-T,t-T[\\
f(0) g(0) \, \partial^2_{\xi\xi} X^{t,\eta(0)}_{T+\rho}& \rho\in [t-T,0].
\end{array}
\right.
\end{split}
\ee
\end{enumerate}
\end{lem}
In the  remark below
 we express Lemma \ref{lem L1re}  in the case of the functional Brownian flow.
\begin{rem}  \label{der X Brown}
When $\sigma(t,x) \equiv \sigma$ is a constant,
by (\ref{defXF})
 the following holds.
\begin{enumerate}
\item
\be \label{eq_der_X_Brown}
\partial_\xi X^{s,\xi}_{t}=1 \quad \mbox{and}\quad \partial^{2}_{\xi\xi} X^{t,\xi}_{s}=0.
\ee
\item  By \eqref{eq_der_X_Brown} the derivatives given by \eqref{eq DerY} for the functional Brownian flow reduce to
\be			\label{eq DerYBr}
\begin{split}
\rho\mapsto \int_{[-T,0]}D_{dx} Y^{t,\eta}_{T}(\rho)f(x)
&=
\left\{
\begin{array}{ll}  f(\rho+T-t) & \rho\in [-T,t-T[\\
f(0)	& \rho\in [t-T,0]
\end{array}
\right.\\
 \mbox{ and }
\\
\rho\mapsto \int_{[-T,0]^2}D^2_{dy\,dx} Y^{t,\eta}_{T}(\rho)f(x)g(y)
&= 0 \quad \quad \quad \quad \rho\in [-T,0].
\end{split}
\ee
\end{enumerate}
\end{rem}
%

%
%
%

\subsection{The existence result for smooth Fr\'echet terminal condition}			
In this section, Theorem \ref{thm 2 derivate u}  states the existence result and Fr\'echet regularity of the solution of the infinite dimensional PDE \eqref{kolmogorov} when $\sigma$ is constant and $H$ is $C^3(C([-T,0]))$.
%
%
%
In particular we will give conditions on the function $H$ such that $u$
defined in \eqref{eq u} solves the PDE stated on \eqref{kolmogorov}.
Those conditions are reasonable but they are however not optimal.
%
%
\begin{thm}			\label{thm 2 derivate u}
Let $H\in C^{3}\left(C([-T,0])\right)$ such that $D^{3}H$ has polynomial growth (for instance bounded).
Let $u$ be defined by $u(t,\eta)=\mathbb{E}\left[ H\big( Y_{T}^{t,\eta} \big)
\right],  t \in [0,T], \eta \in C([-T,0]).$
\begin{enumerate}
\item[1)]
Then $u\in C^{0,2}([0,T]\times C([-T,0]))$. 
\item[2)] Suppose moreover the following for every $\eta \in C([-T,0])$:
\begin{itemize}
\item[i)] The measure $D_{dx}H(\eta)$ is Lebesgue absolutely continuous. We will denote $x\mapsto D_{x}H(\eta)$ its density and we suppose that $DH(\eta)\in H^{1}([-T,0])$, i.e. function $x\mapsto D_{x}H(\eta)$ is in $H^{1}([-T,0])$.
\item[ii)]
$DH$ has polynomial growth in $H^{1}([-T,0])$, i.e.
there is $p\geq 1$ such that
\begin{equation}			\label{eq HS11}
\eta\mapsto \| DH(\eta) \|_{H^{1}}
\leq
const\left(  \|\eta\|^{p}_{\infty} +1\right)	\; .		
\end{equation}
In particular
\begin{equation}			\label{eq HS11bis}
\sup_{t \in [-T,0]} \vert  D_x H(\eta) \vert \le const
 \left(  \|\eta\|^{p}_{\infty} +1\right)
\leq
const\left(  \|\eta\|^{p}_{\infty} +1\right)	\; .
\end{equation}
\item[iii)] The map
\begin{equation}		\label{eq HS2}
\eta \mapsto DH(\eta)
\textrm{ considered as }
C([-T,0])\rightarrow H^{1}([-T,0]) \textrm{ is continuous.}
\end{equation}
\end{itemize}

\noindent
Then $u\in C^{1,2}([0,T]\times C([-T,0]))$ and 
 $u$ is a classical solution of \eqref{kolmogorov}  in $C([-T,0])$, i.e. $u$ solves
\begin{equation}		 \label{kolmogorovF}
\begin{dcases}
\partial_{t}u(t,\eta)+\int_{]-t,0]} D^{\perp}_{dx}u(t,\eta)\,d^{-}\eta(x) +\frac{1}{2}\sigma^2\langle D^2u(t,\eta),\1_{\{0\}}\otimes^2 \rangle =0 & \\
u(T,\eta)=H(\eta). &  \\
\end{dcases}
\end{equation}
\end{enumerate}
\end{thm}
\begin{rem}
Contrarily to the (non-degenerate) situation of Section \ref{Toy model},
Theorem \ref{thm 2 derivate u} holds even when $\sigma = 0$.
In that case one gets a first-order equation;
the regularity on $H$ could be relaxed but we are not specifically
interested in this refinement.
\end{rem}

\begin{rem}		\label{rm PMLOP}
\begin{enumerate}
\item
Assumption \eqref{eq HS11} implies
in particular that $DH$ has polynomial growth in $C([-T,0])$, i.e.
there is $p\geq 1$ such that
\begin{equation}			\label{eq HS1}
\eta\mapsto \sup_{x\in [-T,0]}\vert D_{x}H(\eta) \vert=
\| DH(\eta) \|_{\infty}
\leq
const\left(  \|\eta\|^{p}_{\infty} +1\right).			
\end{equation}
Indeed it is well-known that
$H^{1}([-T,0]) \hookrightarrow  C([-T,0])$ and for a function $f\in H^{1}$ it holds $\|f\|_{\infty}\leq const \| f \|_{H^{1}}$.
\item By a Taylor's expansion, given for instance by
 Theorem 5.6.1 in \cite{cartan}, the fact that $D^{3}H$ has polynomial growth implies that
$H$, $DH$ and $D^{2}H$ have also polynomial growth in $C([-T,0])$.
\item $Du(t,\eta)$, $D^{2}u(t,\eta)$ and $\partial_{t}u(t,\eta)$ will be explicitly expressed in term of $H$ at \eqref{eq Du}, \eqref{eq D2u} and \eqref{eq dtu}.
\end{enumerate}
\end{rem}

\begin{proof}

By expression \eqref{eq u} it is obvious that $u(T,\eta)=H(\eta)$. \\
\emph{Proof of 1).}

$\bullet$ {\bf Continuity of function $u$ with respect to time $t$.}\\
We consider a sequence $(t_{n})$ in $[0,T]$ such that
 $t_{n}\xrightarrow [n\rightarrow \infty]{} t_{0}$.
By Assumption,
 $H\in C^{0}(C([-T,0]))$.
Consequently, by Proposition \ref{Pdef flow}
\begin{equation}		\label{eq H contin}
H\big( Y_{T-t_{n}}^{0,\eta}\big)
\xrightarrow [n\rightarrow \infty]{a.s.}
H\big( Y_{T-t_{0}}^{0,\eta}\big)  \; .
\end{equation}
By Remark \ref{rm PMLOP}.1. $H$ has also polynomial growth, therefore there is $p\geq 1$ such that
\[
\vert H(\zeta) \vert \leq const \; \left(1 + \sup_{x\in[-T,0]} |\zeta(x)|^{p}\right)		\hspace{2cm} \forall \; \zeta\in C([-T,0])   \; .
\]
By \eqref{eq MAGG0},
we observe that
\[
\begin{split}
\vert H(Y^{0,\eta}_{T-t})\vert
&
\leq
const \left( 1+ \left\| Y^{0,\eta}_{T-t} \right\|^{p}_{\infty}  \right)\leq
\\
&
\leq
const \left( 1+ \sup_{x\in [-T,0]}|\eta(x)|^{p}+
\sigma^p \sup_{t\leq T}|W_{t}|^{p} \right)  \; .
\end{split}
\]
By Lebesgue dominated convergence theorem, the fact that
$\sup_{t\leq T}|W_{t}|^{p}$ is integrable and \eqref{eq H contin}, it follows that
\begin{equation}		\label{eq continuity u}
u(t_{n},\eta)=\mathbb{E}\left[ H\big( Y_{T-t_{n}}^{0,\eta}\big) \right]
\xrightarrow [n\rightarrow \infty]{}
\mathbb{E}\left[ H\big( Y_{T-t_{0}}^{0,\eta}\big) \right]=u(t_{0},\eta)		\; .
\end{equation}

$\bullet$ \textbf {First order Fr\'echet derivative.}\\
We express now the derivatives of $u$ with respect to the derivatives of $H$.
We start with
$Du\,:[0,T]\times C([-T,0])\longrightarrow \shm([-T,0])$.
Omitting some details, by integration theory for every $t\in [0,T]$, $u (t, \cdot)$ is of class $C^{1}\apt C([-T,0])\cpt$.
By usual derivation rules for composition we have
$$ D_{dx}H \apt Y^{t,\eta}_T \cpt=\int_{[-T,0]} D_{d \rho} H \apt Y^{t,\eta}_T \cpt D_{dx} Y_T^{t,\eta}(\rho).$$
\be		\label{eq 4}
D_{dx}u\apt t,\eta \cpt=\mathbb{E}\left[ D_{dx}H \apt Y^{t,\eta}_T \cpt \right]=\mathbb{E}\left[ \int_{[-T,0]} D_{d \rho} H \apt Y^{t,\eta}_T \cpt D_{dx} Y_T^{t,\eta}(\rho) \right].
\ee
We compute explicitly \eqref{eq 4} using the expression \eqref{eq DerY2}.
 Integrating with respect to $\rho$ (for a fixed $x$), we obtain the following.

\be\label{E422bis}
\begin{split}
D_{dx}u\apt t,\eta \cpt&
=\mathbb{E}\left[ \int_{[-T,t-T[} D_{d \rho} H \apt Y^{t,\eta}_T \cpt D_{dx} Y_T^{t,\eta}(\rho) \right]+\mathbb{E}\left[ \int_{[t-T,0]} D_{d \rho} H \apt Y^{t,\eta}_T \cpt D_{dx} Y_T^{t,\eta}(\rho) \right]\\
&
=\mathbb{E}\left[ \int_{[-T,t-T[}  D_{d \rho} H \apt Y^{t,\eta}_T \cpt \delta_{\rho+T-t}(dx)  \right]
+ \mathbb{E}\left[ \int_{[t-T,0]}  D_{d  \rho} H \apt Y^{t,\eta}_T \cpt  \; \right] \delta_{0}(dx).
\end{split}
\ee
Consequently
\begin{equation}		\label{eq Du}
D_{dx}u(t,\eta)=D^{\perp}_{dx}u\,(t,\eta) +
D^{\delta_{0} }u\,(t,\eta) \delta_{0}(dx),
\end{equation}
where
\begin{equation}		\label{eq Dperpu}
D^{\perp}_{dx}u(t,\eta)=
\mathbb{E}\left[ D_{dx-T+t}H\big(Y_{T}^{t,\eta} \big)\right]\1_{[-t,0[}(x)
\end{equation}
and
\begin{equation}	\label{eq D0u}
D^{\delta_{0} }u\,(t,\eta)=
\mathbb{E}\left[  \int_{[t-T,0]}D_{d\rho}H\big( Y_{T}^{t,\eta}\big)\right].
\end{equation}
%
%
Indeed  the first addend
 $D^{\perp}_{dx}u\,(t,\eta)$ of  \eqref{eq Du}, i.e. expression \eqref{eq Dperpu}
comes from \eqref{E422bis}, using
 the fact that $\delta_{\rho+T-t}(dx) =\delta_{dx-T+t}(d\rho)$ and integrating with respect to $\rho$.
The continuity of $(t,\eta) \mapsto D_{dx}u(t,\eta)$ in
\eqref{eq Du} can be justified since the functions
$[0,T]\times C([-T,0])\rightarrow \R$,
 $(t,\eta)\mapsto D^{\delta_{0} }u\,(t,\eta)$
 and
function $[0,T]\times C([-T,0])\rightarrow\mathcal{M}([-T,0])$
 defined by $(t,\eta)\mapsto D^{\perp} u(t,\eta)$ are both continuous.
 The latter fact follows from the fact that $H\in C^{1}(C([-T,0]))$, $DH$ with polynomial growth, \eqref{eq MAGG0}, \eqref{eq MAGG1}, the fact that for any given Brownian motion
$\bar{W}$, $\sup_{x\leq T}\vert \bar{W_{x}}\vert$ has all moments and finally the Lebesgue dominated convergence theorem.

$\bullet$ \textbf{Second order Fr\'echet derivative.}\\
We discuss the second derivative
$$D^{2}u\,:[0,T]\times C([-T,0])\longrightarrow (C([-T,0])\hat{\otimes}_{\pi}C([-T,0]))^{\ast}\cong \mathcal{B}(C([-T,0]),C([-T,0])).$$
For every fixed $(t,\eta)$ we get
\begin{equation}		\label{eq D2u1}
\begin{split}
D^{2}_{dx, dy}u(t,\eta)&= \mathbb{E}\Big[ D^2_{dy-T+t,dx-T+t}H\big(Y_{T}^{t,\eta} \big) \1_{[-t,0[}(x)\otimes \1_{[-t,0[}(y)\Big]  \, +\\
&
+\mathbb{E}\left[ D_{dx-T+t} \langle D H\big(  Y_{T}^{t,\eta} \big), \1_{[t-T,0]}\rangle \right]\1_{[-t,0[}(x)\,\delta_{0}(dy)+\\
&
+\mathbb{E}\left[  D_{dy-T+t} \langle D H\big(  Y_{T}^{t,\eta} \big), \1_{[t-T,0]}\rangle  \right]\1_{[-t,0[}(y) \,\delta_{0}(dx)+\\
&
+\mathbb{E}\left[ \langle D^{2}H\big(  Y_{T}^{t,\eta} \big) , \1_{[t-T,0]}\otimes  \1_{[t-T,0]} \rangle  \right]\delta_{0}(dx)\,\delta_{0}(dy)  \; .
\end{split}
\end{equation}
It is possible to show that all the terms in the first and the second derivative are well defined and continuous using similar techniques used in the first part of the proof.
We omit these technicalities for simplicity.
\begin{rem}
For illustration, if $D^2H$ is an absolutely continuous Borel measure
on $[-T,0]^2$ with density
 $D^2_{x,y}H = D_x D_y H$, we obtain the following
\begin{equation}		\label{eq D2u}
\begin{split}
D^{2}_{dx, dy}u(t,\eta)&= \mathbb{E}\Big[ D_{y-T+t} D_{x-T+t}H\big(Y_{T}^{t,\eta} \big) \Big]  \,
\1_{[-t,0[}(x) \, \1_{[-t,0[}(y)dx\,dy+\\
&
+\mathbb{E}\left[ \int_{t-T}^{0}  D_{s}D_{x-T+t} H\big(  Y_{T}^{t,\eta} \big)ds  \right]\1_{[-t,0[}(x) dx\,\delta_{0}(dy)+\\
&
+\mathbb{E}\left[ \int_{t-T}^{0}  D_{y-T+t} D_{s} H\big(  Y_{T}^{t,\eta} \big)ds  \right]\1_{[-t,0[}(y) dy\,\delta_{0}(dx)+\\
&
+\mathbb{E}\left[ \int_{[t-T,0]^{2}} D_{s_{1}}D_{s_{2}} H\big(  Y_{T}^{t,\eta} \big) ds_{1}\,ds_{2} \right]\delta_{0}(dx)\,\delta_{0}(dy)  \; .
\end{split}
\end{equation}
\end{rem}

\emph{Proof of 2)}
\begin{rem}		\label{remFBG}
Under hypothesis 2)
 we  remark the following.
\begin{enumerate}
\item The right-hand side of \eqref{eq Dperpu} is absolutely continuous
in $x$.
In other words $D^\perp_{dx} u(t,\eta)=D^{ac}_{x}u\,(t,\eta) dx$
and
{\footnotesize{
\begin{equation}		\label{eq Dacu}
D^{ac}_{x}u(t,\eta)
=\mathbb{E}\left[ D_{x-T+t}H\big(Y_{T}^{t,\eta} \big)\right]\1_{[-t,0[}(x)=
\left\{
\begin{array}{ll}
0													&	x\in [-T,-t[\\
\mathbb{E}\left[ D_{x-T+t}H\big(Y_{T}^{t,\eta} \big) \right]			&	x\in\; [-t,0[ \ .
\end{array}
\right.
\end{equation}
}}
\item In particular by item ii), $x \mapsto D_x H(\eta)$ belongs to $H^1$, so
 it has bounded variation. Therefore the deterministic forward integral in
 \eqref{kolmogorov} exists because of Proposition \ref{prop A}
 and it can be expressed through \eqref{eq BACDI2}.
We will denote by $D'H(\eta)$ the derivative in $x$ of function
$x\mapsto D_{x}H(\eta)$, where $D_x H(\eta)$ is the density of the measure $D_{dx}H(\eta)$  for every fixed $\eta$.
Since $x\mapsto D_{x}H(\eta)$
 is absolutely continuous
then, by \eqref{DetIAc} we have
\begin{equation} \label{E111}
 \int_{]-t,0]}D_{dx-T+t}H\left( Y_{T}^{t,\eta}\right)d^- \eta(x) =
\int_{]-t,0]}D_{x-T+t}H\left( Y_{T}^{t,\eta}\right)d^- \eta(x).
\end{equation}
Previous deterministic integral exists because
$x \mapsto D_x H(\eta)$ has bounded variation
and by Proposition \ref{prop A}
it equals
$$
-  D_{-T}H\left(Y_{T}^{t,\eta}\right)\eta(-t)
+ D_{t-T}H\left( Y_{T}^{t,\eta} \right) \eta(0)
-
 \int_{-t}^{0}D'_{x-T+t} H\left( Y_{T}^{t,\eta} \right)\eta(x)dx.
$$
\end{enumerate}
\end{rem}
$\bullet$ \textbf{ Derivability with respect to time $t$.} \\
Let $t \in [0,T], \eta \in C([-T,0)]$.
We will show that
\begin{equation} \label{EQuasiFinal}
\partial_t u(t,\eta)=
- \mathbb{E}\left[\int_{]-t,0]}D_{x-T+t} H\left( Y_{T}^{t,\eta} \right) d^- \eta(x)
 + \frac{\sigma^2}{2}\langle D^2H\left( Y_{T}^{t,\eta} \right)
 \ , \ 1_{]t-T,0]}\otimes^2 \rangle  \right] .
 \end{equation}
We need to consider
$\epsilon$ such that $t+\epsilon
\in [0,T]$ and evaluate the limit, if it exists, of
\begin{equation}		\label{eq approx partial t}
\frac{u(t+\epsilon,\eta)-u(t,\eta) }{\epsilon},
\end{equation}
when $\epsilon \rightarrow 0$. Without restriction
of  generality we will suppose  here $\epsilon > 0$; the case $\epsilon < 0$ would bring  similar
calculations.

 The flow property \eqref{eq: flow property}
gives $Y_{T}^{t,\eta}=Y_{T}^{t+\epsilon, Y_{t+\epsilon}^{t,\eta}}$,
which allows to write
\begin{equation}		\label{eq AQ}
u(t,\eta)=\mathbb{E}\left[   H\big( Y_{T}^{t+\epsilon,Y_{t+\epsilon}^{t,\eta}}\big)   \right]	\; .
\end{equation}
We go on with the evaluation of the  limit  of \eqref{eq approx partial t}.
By \eqref{eq AQ} and by differentiability of $H$ in $C([-T,0])$ we have
{\footnotesize{
\begin{equation}		\label{eq AQW}
\begin{split}
H\big(Y_{T}^{t+\epsilon,\eta} \big)& -H\big( Y_{T}^{t+\epsilon,Y_{t+\epsilon}^{t,\eta}}\big)
=
\langle DH\left(Y_{T}^{t,\eta} \right),Y_{T}^{t+\epsilon,\eta}-Y_{T}^{t+\epsilon,Y_{t+\epsilon}^{t,\eta}}\rangle +
\\
&
+
\int_{0}^{1}\langle DH\left( \alpha Y_{T}^{t+\epsilon,\eta}+(1-\alpha) Y_{T}^{t+\epsilon,Y_{t+\epsilon}^{t,\eta } }  \right)-DH\left(Y_{T}^{t,\eta} \right),
Y_{T}^{t+\epsilon,\eta}-Y_{T}^{t+\epsilon,Y_{t+\epsilon}^{t,\eta}}
\rangle d\alpha
=\\
%
&
=\int_{[-T,0]} D_{dx}H\left( Y_{T}^{t,\eta} \right)
\left(  Y_{T}^{t+\epsilon,\eta}(x) -Y_{T}^{t+\epsilon,Y_{t+\epsilon}^{t,\eta}}(x)  \right)+S(\epsilon,t,\eta)\ ,
\\
\end{split}
\end{equation}
}}
where
{\footnotesize{
\[
S(\epsilon,t,\eta)=
\int_{0}^{1}\langle DH\left( \alpha Y_{T}^{t+\epsilon,\eta}+(1-\alpha) Y_{T}^{t+\epsilon,Y_{t+\epsilon}^{t,\eta } }  \right)-DH\left(Y_{T}^{t,\eta} \right),
Y_{T}^{t+\epsilon,\eta}-Y_{T}^{t+\epsilon,Y_{t+\epsilon}^{t,\eta}}
\rangle d\alpha			\; .
%
%
%
\]
}}
Setting $\gamma=Y_{t+\epsilon}^{t,\eta}$, we need to evaluate
\begin{equation}	\label{eq EE}
Y_{T}^{t+\epsilon,\eta}(x)-Y_{T}^{t+\epsilon,\gamma}(x)
\quad\quad x\in [-T,0]			\; .
\end{equation}
\eqref{eq EE} gives
{\footnotesize{
\begin{equation}		\label{eq EEE}
Y_{T}^{t+\epsilon,\eta}(x)-Y_{T}^{t+\epsilon,\gamma}(x) =
\left\{
\begin{array}{ll}
\eta(x+T-t-\epsilon)-\gamma(x+T-t-\epsilon)	&	
x\in [-T,t-T+\epsilon[\\
\eta(0)-\gamma(0) =-\sigma( W_{t+\epsilon}(0)+W_{t}) 		&				x\in [t-T+\epsilon,0] \ ,
\end{array}
\right.
\end{equation}
}}
because $\gamma(0)=Y_{t+\epsilon}^{t,\eta}(0)=\eta(0)+
\sigma (W_{t+\epsilon}(0)-W_{t}) $.
Moreover, by \eqref{eq flow}, we have
{\footnotesize{
\[
\gamma(x+T-t-\epsilon)=Y_{t+\epsilon}^{t,\eta}(x+T-t-\epsilon)=
\left\{
\begin{array}{ll}
\eta(x+T-t)					&			x\in[-T,t-T[	  \\
\eta(0)+\sigma(W_{T}(x)-W_{t})			&			x\in [t-T,t-T+\epsilon]   .
\end{array}
\right.
\]
}}
Finally we obtain an explicit expression for \eqref{eq EE}; indeed \eqref{eq EEE} gives
{\footnotesize{
\begin{equation}				\label{eq EAE}
Y_{T}^{t+\epsilon,\eta}(x)-Y_{T}^{t+\epsilon,\gamma}(x) =
\left\{
\begin{array}{ll}
\eta(x+T-t-\epsilon)-\eta(x+T-t)		&	x\in [-T,t-T[\\
\eta(x+T-t-\epsilon)-\eta(0)-\sigma(W_{T}(x)+W_{t}) &	x\in [t-T, t-T+\epsilon[\\
\sigma(W_{t}-W_{t+\epsilon})				&	x\in [t-T+\epsilon,0] \ .
\end{array}
\right.
\end{equation}
}}
Consequently, using \eqref{eq AQ}, \eqref{eq AQW} and \eqref{eq EAE}, the quotient \eqref{eq approx partial t} appears to be the sum of four terms.
\begin{equation}		\label{eq ZA}
\begin{split}
\frac{u(t+\epsilon,\eta)-u(t,\eta) }{\epsilon}&=
\mathbb{E}\left[ \frac{ H\big( Y_{T}^{t+\epsilon,\eta} \big) -  H\big( Y_{T}^{t+\epsilon,Y_{t+\epsilon}^{t,\eta}}\big) }{\epsilon}\right]
= \\
&=
I_{1}(\epsilon,t, \eta)+I_{2}(\epsilon,t, \eta)+I_{3}(\epsilon,t, \eta)+\frac{1}{\epsilon}\mathbb{E}\left[ S(\epsilon,t, \eta)\right],
\end{split}
\end{equation}
where
\[
\begin{split}
I_{1}(\epsilon,t, \eta)&=\mathbb{E}\left[ \int_{-T}^{t-T} D_{x}H\left( Y_{T}^{t,\eta}\right)\frac{\eta(x+T-t-\epsilon)-\eta(x+T-t) }{\epsilon}  dx \right]=
\\
&
=-\mathbb{E}\left[ \int_{-t}^{0} D_{x-T+t }H\left( Y_{T}^{t,\eta}\right) \frac{\eta(x)-\eta(x-\epsilon) }{\epsilon}  dx \right]\\
I_{2}(\epsilon,t, \eta)&=\mathbb{E}\left[ \int_{t-T}^{t-T+\epsilon}D_{x}H\left( Y_{T}^{t,\eta}\right) \frac{\eta(x+T-t-\epsilon)-\eta(0)-
\sigma(W_{T}(x)+W_{t})  }{\epsilon} dx \right] \\
& -\mathbb{E}\left[ \int_{t-T}^{t-T+\epsilon}D_{x}H\left( Y_{T}^{t,\eta}\right)   \frac{W_{t}-W_{t+\epsilon}}{\epsilon} dx \right]\\
&= \mathbb{E}
\left[
\int_{t-T}^{t-T+\epsilon}
D_{x}H\left( Y_{T}^{t,\eta}\right)\frac{\eta(x+T-t-\epsilon)-\eta(0)-
\sigma(W_{T}(x)+W_{t+\epsilon})}{\epsilon}
dx \right]
\end{split}
\]
\[
\begin{split}
I_{3}(\epsilon,t, \eta)&=\mathbb{E}\left[ \int_{t-T}^{0}D_{x}H\left( Y_{T}^{t,\eta}\right)   \frac{\sigma(W_{t}-W_{t+\epsilon})}{\epsilon}  dx \right] \\
\end{split}
\]
and $\mathbb{E}\left[ S(\epsilon,t, \eta)\right]$ is equal to
\begin{equation}		\label{eq ER}
\begin{split}
\int_{0}^{1}
\mathbb{E}
\left[
\int_{-T}^{0}\left(  D_{x}H\left( \alpha Y_{T}^{t+\epsilon,\eta}+(1-\alpha) Y_{T}^{t+\epsilon,Y_{t+\epsilon}^{t,\eta } }  \right)-D_{x}H\left(Y_{T}^{t,\eta} \right) \right) \cdot \right.\\
\hspace{5cm}
\left.
\cdot
\left(  Y_{T}^{t+\epsilon,\eta}(x) -Y_{T}^{t+\epsilon,Y_{t+\epsilon}^{t,\eta}}(x)  \right)dx\right]
d\alpha			\; .
\end{split}
\end{equation}
$\bullet$
We will prove that
\begin{equation} \label{I123}
I_{1}(\epsilon,t, \eta)   \xrightarrow[\epsilon\rightarrow 0]{}
I_{1}(t, \eta):=
I_{11}(t, \eta)+I_{12}(t, \eta)+I_{13}(t, \eta),
\end{equation}
where
\[
\begin{split}
I_{11}(t, \eta)&=
\mathbb{E}\left[ D_{-T}H\left(Y_{T}^{t,\eta}\right)\eta(-t)\right] \\
I_{12}(t, \eta)&=
\mathbb{E}\left[  \int_{-t}^{0}D'_{x-T+t} H\left( Y_{T}^{t,\eta} \right)\eta(x)dx\right] \\
I_{13}(t, \eta)&=
- \mathbb{E}\left[ D_{t-T}H\left( Y_{T}^{t,\eta} \right) \eta(0) \right]   \; .
\end{split}
\]
Admitting \eqref{I123},
the additivity
and using \eqref{eq Dacu} in Remark \ref{remFBG} we have

\begin{equation}
  I_1(t,\eta) =  -\mathbb{E}\left[ \int_{]-t,0]}D_{x-T+t} H\left( Y_{T}^{t,\eta} \right) d^- \eta(x)\right].
\end{equation}
It remains to show \eqref{I123}.
In fact $I_{1}(\epsilon,t,\eta)$ can be rewritten as sum of the three terms
\[
\begin{split}
I_{11}(\epsilon,t,\eta)
&
=\mathbb{E}\left[ \int_{-t}^{-t+\epsilon} D_{x-T+t}H\left( Y_{T}^{t,\eta}\right) \frac{\eta(x-\epsilon)}{\epsilon} dx  \right]
\\
I_{12}(\epsilon,t,\eta)
&
=\mathbb{E}\left[ \int_{-t}^{0}\frac{ D_{x+\epsilon-T+t}H\left( Y_{T}^{t,\eta}\right)-D_{x-T+t}H\left( Y_{T}^{t,\eta}\right) }{\epsilon} \eta(x) dx\right]
\\
I_{13}(\epsilon,t,\eta)
&
=-\mathbb{E}\left[ \int_{0}^{\epsilon} D_{x-T+t }H\left( Y_{T}^{t,\eta}\right) \frac{\eta(x-\epsilon)}{\epsilon}  dx \right]			\; .
\\
\end{split}
\]
We can apply the
 dominated convergence theorem.
Since  $\bar{W}$, $\sup_{x\leq T}\vert \bar{W_{x}}\vert$ has all moments and
taking into account \eqref{eq HS1} in Remark \ref{rm PMLOP},
 we get that $I_{1i}(\epsilon,t,\eta)\xrightarrow
 [\epsilon\rightarrow 0]{} I_{1i}(t,\eta)$ for $i=1,2,3$ holds.\\
%
$\bullet$
 $ I_{2}(\epsilon,t,\eta) $ converges to zero when $\epsilon\rightarrow 0$.
Indeed, Cauchy-Schwarz inequality yields
\begin{equation}		\label{eq F1234}
\begin{split}
\vert I_{2}(\epsilon,t,\eta) \vert^{2}&\leq
\frac{1}{\epsilon}\mathbb{E}\left[
\int_{t-T}^{t-T+\epsilon}
D_{x}H\left( Y_{T}^{t,\eta}\right)^{2}dx
\right] \cdot \\
&\hspace{1.2cm}\cdot
\frac{1}{\epsilon}\mathbb{E}\left[
\int_{t-T}^{t-T+\epsilon}
\left( \eta(x+T-t-\epsilon)-\eta(0)-\sigma(W_{T}(x)+W_{t+\epsilon})\right)^{2}dx
\right]		.
\end{split}
\end{equation}
Again, by usual arguments and
again because
 $\sup_{x\leq T}\vert \bar{W_{x}}\vert$ has all moments and
taking into account \eqref{eq HS1} in Remark \ref{rm PMLOP},
it
follows that the first integral converges to $\mathbb{E}\left[D_{t-T}H\left( Y_{T}^{t,\eta}\right)^{2} \right]$ and the second integral to zero.\\
%
$\bullet$ As third step we prove that
\be		\label{TRE2}
I_{3}(\epsilon,t,\eta)  \xrightarrow
[\epsilon\rightarrow 0]{}
- \sigma^2 \mathbb{E}\left[ \langle D^2 H\big(  Y_{T}^{t,\eta} \big) , \1_{]t-T,0]}\otimes^2 \rangle \right]=:I_{3}(t,\eta)  \; .
\ee
For this, we rewrite $I_{3}(\epsilon,t,\eta)$ using \eqref{Wbar}, i.e. $W_{t+\epsilon}-W_t=\overline W_{\epsilon}$ and the Skorohod integral to obtain
{\footnotesize{
\begin{equation} \label{eq GVB}
\begin{split}
I_{3}(\epsilon,t,\eta)&
=-\sigma \mathbb{E}\left[ \int_{t-T}^{0}D_{x}H\left( Y_{T}^{t,\eta}\right)
  \frac{W_{t+\epsilon}-W_{t}}{\epsilon}  dx \right] =-
\frac{\sigma}{\epsilon} \mathbb{E}\left[ \int_{t-T}^{0}D_{x}H\left( Y_{T}^{t,\eta}\right)    dx \ \cdot \ \overline W_{\epsilon} \right]\\
&
=-\frac{\sigma}{\epsilon}\mathbb{E}\left[ \int_{t-T}^{0}D_{x}H\left( Y_{T}^{t,\eta}\right)  dx \ \cdot \  \int_{0}^{\epsilon} \delta \overline W_{r} \right]
= -\frac{\sigma}{\epsilon}\mathbb{E}\left[\mathcal{Z}\cdot \int_{0}^{\epsilon}\delta  \overline W_{s}  \right],			
\end{split}
\end{equation}
}}
where $\mathcal{Z}:=\langle DH(Y_{T}^{t,\eta}),\1_{]t-T,0]}\rangle$.
Denoting by the deterministic function
$\mathcal{Y}:=1_{]t-T,0]}(x)$,
using Proposition \ref{propP1112} with $n=1$,
it follows that
 $\mathcal{Z}=\langle DH(Y_{T}^{t,\eta}), \mathcal{Y}\rangle $ belongs to $\mathbb{D}^{1,2}$
and
\be	\label{RFT}
D^{m}_{r}\mathcal{Z}=\sigma
\langle D^2H\left(Y_{T}^{t,\eta}\right), \1_{]t-T,0]}(x)\otimes \1_{]r-T+t,0]}(y) \rangle.
\ee
By integration by parts on Wiener space, expression \eqref{RFT}, Fubini's theorem with respect to $r$ and $y$,
\eqref{eq GVB} gives
{\footnotesize{
\be		\label{THJ1}
\begin{split}
I_{3}(\epsilon,t,\eta)=
-\frac{\sigma}{\epsilon}\mathbb{E}\left[ \int_{0}^{\epsilon}D^{m}_{r}\mathcal{Z}\,dr\right]
&
=-  \frac{\sigma^2}{\epsilon} \mathbb{E}\left[ \int_{0}^{\epsilon}\langle D^2 H\left( Y_{T}^{t,\eta}\right) , \1_{]r-T+t,0]}(x)\otimes \1_{]t-T,0]}(y) \rangle \;dr \right]
\\
&
=
-  \frac{\sigma^2}{\epsilon}\mathbb{E}\left[ \langle D^2 H\left( Y_{T}^{t,\eta}\right) , \int_{0}^{\epsilon}  \1_{]r-T+t,0]}(x) dr \otimes \1_{]t-T,0]}(y) \rangle \right]
\\
&
=
-  \frac{\sigma^2}{\epsilon}\mathbb{E}\left[ \langle D^2 H\left( Y_{T}^{t,\eta}\right) , \int_{t}^{t+\epsilon}  \1_{]z-T,0]}(x) dz \otimes \1_{]t-T,0]}(y) \rangle \right],
\end{split}
\ee
}}
where the latter equality comes replacing $z:=r+t$ in the integral.

Observing that
\begin{equation}	\label{eq45GTB}
\begin{split}
&\int_{t}^{t+\epsilon}  \1_{]z-T,0]}(x) dz=\int_{t}^{t+\epsilon}  \1_{[0,x+T[}(z) dz\\
&\hspace{1cm}=
\left\{ \begin{array}{ll}
\int_{t}^{t+\epsilon}  0 dz=0 & x \leq t-T  \  \Leftrightarrow \  x+T\leq t \\
\int_{t}^{t+\epsilon}    \1_{[0,x+T[}(z) dz=x-t & x \in ]t-T, t-T+\epsilon] \   \Leftrightarrow \  x+T \in ]t, t+\epsilon]   \\
\int_{t}^{t+\epsilon}  1 dz=\epsilon & x \in ]t-T+\epsilon,0] \  \Leftrightarrow \ x+T\in ] t+\epsilon , T],
\end{array}\right.
\end{split}
\end{equation}
we get
\[
\frac{1}{\epsilon} \int_{t}^{t+\epsilon}  \1_{]z-T,0]}(x) dz = \1_{]t-T+\epsilon,0]}(x)+\frac{(x-t)}{\epsilon}\1_{]t-T, t-T+\epsilon]}(x) \ .
\]
Previous expression is bounded by $1$. Moreover it converges pointwise
 to $\1_{]t-T,0]}(x)$ as $\epsilon \downarrow 0$. By Remark  \ref{rm PMLOP} item 1., the fact that $D^2H$ has polynomial growth and that for any given Brownian motion
$\bar{W}$, $\sup_{x\leq T}\vert \bar{W_{x}}\vert$ has all moments and finally the
Lebesgue dominated convergence theorem we conclude that \eqref{THJ1}
 converges to $I_{3}(t,\eta)$, i.e.
\[
I_{3}(t,\eta)
= - \sigma^2 \mathbb{E}\left[ \langle D^2H\left( Y_{T}^{t,\eta}\right), \1_{]t-T,0]}(x)\otimes\1_{]t-T,0]}(y) \rangle \right].
\]
So the convergence \eqref{TRE2} is established.\\
$\bullet$ We study now the term $\frac{1}{\epsilon} \mathbb{E}\left[ S(\epsilon,t, \eta)\right]$ in \eqref{eq ER}.

By Lemma \ref{lemma 1}, we get the a.s. equality
$Y_{T}^{t,\eta}= Y_{T}^{t+\epsilon,Y_{t+\epsilon}^{t,\eta } } $.
Using \eqref{eq EAE}
 and the fact that $H\in C^{3}\left( C([-T,0])\right)$,
\eqref{eq ER} can be rewritten as the sum of the terms
\[
\begin{split}
A_{1}(\epsilon,t,\eta)
&
= \int_{0}^{1}
\mathbb{E}
\left[
\int_{-T}^{t-T}\left(  D_{x}H\left( \alpha Y_{T}^{t+\epsilon,\eta}+(1-\alpha) Y_{T}^{t,\eta} \right)-D_{x}H\left(Y_{T}^{t,\eta} \right) \right) \right. \cdot\\
&\hspace{4cm}
\left. \cdot
\frac{ \eta(x+T-t-\epsilon )-\eta(x+T-t) }{\epsilon} dx\right] d\alpha
\\
A_{2}(\epsilon,t,\eta)
&
=
\int_{0}^{1}
\mathbb{E}
\left[
\int_{t-T}^{t-T+\epsilon}\left(  D_{x}H\left( \alpha Y_{T}^{t+\epsilon,\eta}+(1-\alpha) Y_{T}^{t,\eta}  \right)-D_{x}H\left(Y_{T}^{t,\eta} \right) \right) \right. \cdot
\\
&
\hspace{2.5cm}
\left. \cdot
\frac{\eta(x+T-t-\epsilon)-\eta(0)- \sigma W_{T}(x)+\sigma
W_{t+\epsilon}}{\epsilon} dx\right]d\alpha
\\
A_{3}(\epsilon,t,\eta)&
= A_{31}(\epsilon,t,\eta)+A_{32}(\epsilon,t,\eta)+A_{33}(\epsilon,t,\eta)+A_{34}(\epsilon,t,\eta) \; ,
\\
\end{split}
\]
where
{{\footnotesize{
\[
\begin{split}
A_{31}(\epsilon,t,\eta)
&=
\frac{\sigma^2}{2}
\mathbb{E}\left[ \langle D^{2}H\left( Y_{T}^{t,\eta} \right), \1_{]t-T+\epsilon,0]}\otimes\1_{]t-T+\epsilon,0]}\rangle \cdot \frac{(W_{t}-W_{t+\epsilon})^{2}}{\epsilon}\right] \, ,\\
A_{32}(\epsilon,t,\eta)
&=\sigma^2
\int_{0}^{1}
\mathbb{E}\left[
\langle \left( D^{2}H\left( \alpha Y_{T}^{t+\epsilon,\eta}+(1-\alpha) Y_{T}^{t,\eta}  \right) -D^{2}H\left( Y_{T}^{t,\eta}\right) \right),
\1_{]t-T+\epsilon,0]^{2}}\rangle\cdot \right.\\
&\left. \hspace{7cm} \cdot \frac{(W_{t}-W_{t+\epsilon})^{2}}{\epsilon}\right]d\alpha \, ,\\
A_{33}(\epsilon,t,\eta)
&
= \sigma
\int_{0}^{1}
\mathbb{E}\left[
\langle \left( D^{2}H\left( \alpha Y_{T}^{t+\epsilon,\eta}+(1-\alpha) Y_{T}^{t,\eta}  \right) -D^{2}H\left( Y_{T}^{t,\eta}\right) \right) \right. \\
&\left.
\frac{\eta(y+T-t+\epsilon)-\eta(y+T-t)}{\epsilon} \1_{]t-T+\epsilon,0]}  (x)\otimes  \1_{[-T,t-T]}(y) \rangle \cdot  (W_{t}-W_{t+\epsilon})  \right] d\alpha  \,,
\\
A_{34}(\epsilon,t, \eta)
&
= \sigma
\int_{0}^{1}
\mathbb{E}\left[
\langle \left( D^{2}H\left( \alpha Y_{T}^{t+\epsilon,\eta}+(1-\alpha) Y_{T}^{t,\eta}  \right) -D^{2}H\left( Y_{T}^{t,\eta}\right) \right), \right. \\
& \left.  \frac{\eta(y+T-t-\epsilon)-\eta(0)-
    \sigma(W_{T}(y)+W_{t+\epsilon})}{\epsilon}\1_{]t-T+\epsilon,0]}(x)\otimes \1_{]t-T,t-T+\epsilon]}(y) \rangle \cdot \right. \\
    & \left. \hspace{7cm}  (W_{t}-W_{t+\epsilon})  \right] d\alpha \, . \\
\\
\end{split}
\]
}}}
$\bullet$
Similarly to $I_{1}(\epsilon,t,\eta)$, the term $A_{1}(\epsilon,t,\eta)$
 can be decomposed into the sum of  terms given below.
{ \scriptsize{
\[
\begin{split}
A_{11}(\epsilon,t,\eta)
&
=\mathbb{E}\left[
\int_{0}^{1}\int_{-t}^{-t+\epsilon}D_{x-T+t}H\left(  \alpha Y_{T}^{t+\epsilon,\eta}+(1-\alpha) Y_{T}^{t,\eta} \right)-D_{x-T+t}H\left( Y_{T}^{t,\eta}\right)\frac{\eta(x-\epsilon)}{\epsilon}dx\; d\alpha
\right],
\\
&
\\
A_{12}(\epsilon,t,\eta)
&
=
\mathbb{E}\left[
\int_{0}^{1}\int_{-t}^{0}\frac{D_{x+\epsilon-T+t}H\left(  \alpha Y_{T}^{t+\epsilon,\eta}+(1-\alpha) Y_{T}^{t,\eta  }  \right)
-D_{x-T+t}H\left(  \alpha Y_{T}^{t+\epsilon,\eta}+(1-\alpha) Y_{T}^{t,\eta  }  \right) }{\epsilon} \cdot \right. \\
&\left. \hspace{7 cm}
\cdot \eta(x)dx\,d\alpha  \right] \\
&
-
\mathbb{E}\left[
\int_{0}^{1}\int_{-t}^{0}\frac{D_{x+\epsilon-T+t}H\left( Y_{T}^{t,\eta}\right) -D_{x-T+t}H\left( Y_{T}^{t,\eta}\right)  }{\epsilon}\eta(x)dx\,d\alpha
\right],
\\
&
\\
A_{13}(\epsilon,t,\eta)
&
=-\mathbb{E}\left[
\int_{0}^{1}\int_{-\epsilon}^{0} D_{x-T+t}H\left(  \alpha Y_{T}^{t+\epsilon,\eta}+(1-\alpha) Y_{T}^{t,\eta}  \right)-D_{x-T+t}H\left( Y_{T}^{t,\eta}\right) \frac{\eta(x-\epsilon)}{\epsilon}dx\; d\alpha
\right]  \; .
\\
\end{split}
\]
}}
$\bullet$ We show now that $A_{11}(\epsilon, t, \eta)$ converges to zero.
By Cauchy-Schwarz inequality we have
{\footnotesize{
\[
\begin{split}
\left[A_{11}(\epsilon, t, \eta)\right]^{2}
&
\leq
\int_{-t}^{-t+\epsilon}\frac{\eta^{2}(x-\epsilon)}{\epsilon}dx\; \cdot
\\
&\hspace{-0.5cm}\cdot
\mathbb{E}\left[
\int_{0}^{1}\int_{-t}^{-t+\epsilon}\frac{1}{\epsilon}
\left[ D_{x-T+t}H\left(  \alpha Y_{T}^{t+\epsilon,\eta}+(1-\alpha) Y_{T}^{t,\eta}  \right)-D_{x-T+t}H\left( Y_{T}^{t,\eta}\right)\right]^{2}dx\; d\alpha
\right].
\end{split}
\]
}}
The integral $1/\epsilon \int_{-t}^{-t+\epsilon}\eta^{2}(x-\epsilon)dx$ converges to $\eta^{2}(-t)$ by the finite increments theorem.
By hypotheses \eqref{eq HS2} and \eqref{eq CONVC} we have
\begin{equation}	\label{eq 3456}
\left\|
D H\left(  \alpha Y_{T}^{t+\epsilon,\eta}+(1-\alpha) Y_{T}^{t,\eta}  \right)-D H\left( Y_{T}^{t,\eta}\right)
\right\|_{H^{1}([-T,0])}\xrightarrow [\epsilon \longrightarrow 0 ]{a.s. } 0		\; .
\end{equation}
Because of \eqref{eq 3456}, it follows that
\begin{equation}		\label{eq EC1}
\sup_{x \in [-T,0]}
\vert D_{x} H\left(  \alpha Y_{T}^{t+\epsilon,\eta}+(1-\alpha) Y_{T}^{t,\eta}  \right)-D_{x}  H\left( Y_{T}^{t,\eta}\right) \vert
\xrightarrow [\epsilon \longrightarrow 0 ]{a.s. } 0		\quad\quad \forall \; x \in [-T,0]  \; .
\end{equation}
\eqref{eq EC1} implies that
{\footnotesize{
\[
\int_{0}^{1}\int_{-t}^{-t+\epsilon}\frac{1}{\epsilon}
\left[ D_{x-T+t}H\left(  \alpha Y_{T}^{t+\epsilon,\eta}+(1-\alpha) Y_{T}^{t,\eta}  \right)-D_{x-T+t}H\left( Y_{T}^{t,\eta}\right)\right]^{2}dx\; d\alpha
\xrightarrow [\epsilon \longrightarrow 0 ]{a.s. } 0   \; .
\]
}}
Using \eqref{eq HS1}, \eqref{eq MAGG}, \eqref{eq MAGG1} and the fact that given any Brownian motion
$\bar{W}$, $\sup_{x\leq T}\vert \bar{W_{x}}\vert$ has all moments and
Lebesgue dominated convergence theorem it follows that $A_{11}(\epsilon, t, \eta)$ converges to zero.\\
$\bullet$ Using the same technique we also obtain that $A_{13}(\epsilon, t,\eta)$ converges to zero.\\
$\bullet$ We show that $A_{12}(\epsilon, t, \eta)$ converges to zero.

For every fixed continuous function $\zeta$ we can write
\[
D_{x-T+t+\epsilon}H\left(  \zeta \right) -
D_{x-T+t}H\left(  \zeta \right) =
\int_{x-T+t}^{x+\epsilon-T+t} D'_{u}H\left( \zeta \right) du   \; .
\]
It follows that $A_{12}(\epsilon, t, \eta)$ can be rewritten as
\begin{equation}		\label{eq VB}
\mathbb{E}\left[
\int_{0}^{1}\int_{-t}^{0}\frac{1}{\epsilon} \int_{x-T+t}^{x-T+t+\epsilon}
\left[ D'_{u}H\left(  \alpha Y_{T}^{t+\epsilon,\eta}+(1-\alpha) Y_{T}^{t,\eta  }  \right) -
 D'_{u}H\left(  Y_{T}^{t,\eta  }  \right)
\right] \eta(x)\,
du\,
dx\,
d\alpha  \right]  \; .
\end{equation}
Taking the absolute value and considering the fact that $|\eta(x)|\leq \|\eta\|_{\infty}$ we obtain
{\scriptsize{
\[
\left|
A_{12}(\epsilon,t,\eta)
\right|
\leq
\mathbb{E}\left[
\int_{0}^{1}\int_{-t}^{0}\frac{1}{\epsilon} \int_{x-T+t}^{x-T+t+\epsilon}
\left|
 D'_{u}H\left(  \alpha Y_{T}^{t+\epsilon,\eta}+(1-\alpha) Y_{T}^{t,\eta  }  \right) -
 D'_{u}H\left(  Y_{T}^{t,\eta  }  \right)
\right| \,
du\,
dx\,
d\alpha  \right] \|\eta\|_{\infty}  \; .
\]
}}
By Fubini's theorem it follows
{\footnotesize{
\[
\left|
A_{12}(\epsilon,t,\eta)
\right|
\leq
\mathbb{E}\left
[\int_{0}^{1}
\int_{-T}^{-T+t}
\left|
 D'_{u}H\left(  \alpha Y_{T}^{t+\epsilon,\eta}+(1-\alpha) Y_{T}^{t,\eta  }  \right) -
 D'_{u}H\left(  Y_{T}^{t,\eta  }  \right)
\right| \,
du\,
d\alpha  \right]   \| \eta\|_{\infty} \; .
\]
}}
Now using Cauchy-Schwarz inequality we have
{\footnotesize{
\[
\begin{split}
\left|
A_{12}(\epsilon,t,\eta)
\right|^{2}
&
\leq
T \,
\mathbb{E}\left
[\int_{0}^{1}
\int_{-T}^{-T+t}
\left(
 D'_{u}H\left(  \alpha Y_{T}^{t+\epsilon,\eta}+(1-\alpha) Y_{T}^{t,\eta  }  \right) -
 D'_{u}H\left(  Y_{T}^{t,\eta  }  \right)
\right)^{2} \,
du\,
d\alpha  \right]\; \| \eta\|^{2}_{\infty}
\\
&
\leq
T\,
\mathbb{E}\left
[\int_{0}^{1}
\left\|
 D' H\left(  \alpha Y_{T}^{t+\epsilon,\eta}+(1-\alpha) Y_{T}^{t,\eta  }  \right) -
 D' H\left(  Y_{T}^{t,\eta  }  \right)
\right\|_{L^{2}([-T,0])}^{2}
d\alpha  \right]	\;  \| \eta\|^{2}_{\infty}	\; .
\end{split}
\]
}}
Convergence \eqref{eq 3456} implies in particular
\[
\left\|
 D' H\left(  \alpha Y_{T}^{t+\epsilon,\eta}+(1-\alpha) Y_{T}^{t,\eta  }  \right) -
 D' H\left(  Y_{T}^{t,\eta  }  \right)
\right\|_{L^{2}([-T,0])}^{2} \xrightarrow [\epsilon \longrightarrow 0 ]{a.s. } 0			\; .
\]
Again using \eqref{eq HS1}, \eqref{eq MAGG}, \eqref{eq MAGG1} the fact that given any Brownian motion
$\bar{W}$, $\sup_{x\leq T}\vert \bar{W_{x}}\vert$ has all moments and
Lebesgue dominated convergence theorem we have that $A_{12}(\epsilon, t, \eta)$ converges to zero.\\
$\bullet$ This concludes the proof that
 $A_{1}(\epsilon, t, \eta )$ converges to zero.\\
$\bullet$ Term $A_{2}(\epsilon,t,\eta)$ also converges to zero. In fact  
Cauchy-Schwarz implies that
{\footnotesize{
\[
\begin{split}		
\vert A_{2}(\epsilon,t,\eta) \vert^{2}\leq&
\int_{0}^{1}\frac{1}{\epsilon}
\mathbb{E}
\left[
\int_{t-T}^{t-T+\epsilon}\left(  D_{x}H\left( \alpha Y_{T}^{t+\epsilon,\eta}+(1-\alpha) Y_{T}^{t,\eta}  \right)-D_{x}H\left(Y_{T}^{t,\eta} \right) \right)^{2}dx \right] \cdot
\\
&
\hspace{0.5cm}
 \cdot \frac{1}{\epsilon}
\mathbb{E}
\left[
\int_{t-T}^{t-T+\epsilon}
\left( \eta(x+T-t-\epsilon)-\eta(0)- \sigma W_{T}(x)+ \sigma W_{t+\epsilon}\right)^{2} dx\right]d\alpha \; .
\end{split}
\]
}}
The continuity of $DH$ (see \eqref{eq HS2}), the fact that it has
polynomial growth in the sense of Remark \ref{rm PMLOP}.1.,
\eqref{eq MAGG} and Lebesgue dominated convergence theorem imply that the first expectation converges to zero.
The second expectation converges to zero by the same arguments together with the fact that $\sup_{x\leq T}\vert \bar{W_{x}}\vert$ has all moments.\\
$\bullet$ We show now that $A_{31}(\epsilon,t,\eta)$ converges to
 \begin{equation}		\label{eq FormA31}
\frac{\sigma^2}{2} \mathbb{E}\left[  \langle D^2 H\left( Y_{T}^{t,\eta}\right)  , \1_{]t-T,0]^2} \rangle \right]  =:A_{31}(t,\eta)		\; .
\end{equation}
At this level we need two technical results.
\begin{lem}		\label{lem LA31}
The random variable $B(\epsilon):=\frac{(W_{t+\epsilon}-W_{t})^{2}}{\epsilon}$ weakly converges in $L^{2}(\Omega)$ to $1$ when $\epsilon\rightarrow 0$.
\end{lem}
\begin{proof}
In fact, $\mathbb{E}\left[B(\epsilon)^{2}\right]=3$, so that $\left(B(\epsilon)\right)$ is bounded in
$L^{2}(\Omega)$ . Therefore there exists a subsequence $(\epsilon_{n})$ such that $\left(B(\epsilon_{n}) \right)$
converges weakly to some square integrable variable $B_{0}$. In order to show
 that $B_{0}=1$ and to conclude the proof of the lemma
it is enough to prove that
\begin{equation}
\mathbb{E}\left[ B(\epsilon) \cdot Z \right]
\longrightarrow
\mathbb{E}[Z]
\end{equation}
for any r.v. $Z$ of a dense subset $\mathcal{D}$ of $L^{2}(\Omega)$.
We choose $\mathcal{D}$ and the r.v. $Z$ such that $Z=\mathbb{E}[Z]+\int_{0}^{T}\xi_{s}dW_{s}$
where $(\xi_{s})_{s\in [0,T]}$ is a bounded previsible process.
We have
\[
\mathbb{E}\left[ B(\epsilon)\cdot Z \right]=
\mathbb{E}\left[B(\epsilon)\right]\,\mathbb{E}\left[Z\right]+
\mathbb{E}\left[
\frac{(W_{t+\epsilon}-W_{t})^{2}}{\epsilon}\int_{0}^{T}\xi_{s}dW_{s}
\right].
\]
Since $\mathbb{E}\left[ B(\epsilon) \right]\,\mathbb{E}\left[Z \right]=\mathbb{E}\left[Z \right]$, we only need to show that
\begin{equation}		\label{eq LA31}
\mathbb{E}\left[\frac{(W_{t+\epsilon}-W_{t})^{2}}{\epsilon}\int_{0}^{T}\xi_{s}dW_{s} \right]\xrightarrow[\epsilon\longrightarrow 0] {}0		\; .
\end{equation}
Since $\int_{0}^{T}\xi_{s}dW_{s} $ is a Skorohod integral, integration by parts on Wiener space \eqref{eq INTBPWiener}
implies that
the left-hand side of \eqref{eq LA31} equals
\[
\mathbb{E}\left[
\frac{2}{\epsilon}\int_{0}^{T}\xi_{s}(W_{t+\epsilon}-W_{t})\1_{[t,t+\epsilon]}(s)ds
\right]=
\mathbb{E}
\left[
\frac{1}{\epsilon}\int_{t}^{t+\epsilon}\xi_{s}ds \; (W_{t+\epsilon}-W_{t})
\right] \; ;
\]
this converges to zero since $\xi$ is bounded.
\end{proof}

\begin{lem}\label{lemWS}
Let $H$ be an Hilbert space equipped with a product $\langle \cdot,\cdot\rangle$.
Let $(Z_n)_{n}$ and $(Y_n)_n$ be two sequences in $H$ such that $Z_n$ converges strongly to $Z$ and $Y_n$ converges weakly to $Y$. Then
$\langle Z_n, Y_n\rangle  $ converges to $\langle Z,Y\rangle$.
\end{lem}
\begin{proof}
By Cauchy-Schwarz inequality we obtain
{\footnotesize{
\[
\vert  \langle Z_n, Y_n\rangle  - \langle Z,Y\rangle \vert =\vert \langle Z_n-Z, Y_n\rangle  + \langle Z,Y_n-Y\rangle \vert\leq
\|  Z_n-Z\|_H  \, \|Y_n\|_H +\vert  \langle Z,Y_n-Y\rangle\vert
\xrightarrow[\epsilon\longrightarrow 0]{}0,
\]
}}
since $\|  Z_n-Z\|_H $ goes to zero by the strong convergence hypothesis of $(Y_n)$, $\|Y_n\|_H$ is bounded because weakly convergent and $\langle Z,Y_n-Y\rangle$ goes to zero by definition of weak convergence of $(Y_n)_n$ and the fact that $Z\in H$.
\end{proof}
In order to show the convergence of $2 A_{31}(\epsilon,t,\eta)=
\sigma^2 \mathbb{E}\left[ \mathcal{Z}(\epsilon) \cdot\frac{(W_{t+\epsilon}-W_{t})^{2}}{\epsilon} \right]$ to $2A_{31}(t,\eta)$ we use Lemma \ref{lemWS} setting the Hilbert space $H$ equal
 to $L^2(\Omega)$. We only need to show the strong convergence in $H$ of $\mathcal{Z}(\epsilon)$ to $\mathcal{Z}:=\langle D^{2}H\left( Y_{T}^{t,\eta} \right), \1_{]t-T,0]}\otimes\1_{]t-T,0]}\rangle $.
Taking into account $\1_{]t-T+\epsilon,0]}\otimes^2\rightarrow \1_{]t-T,0]}\otimes^2 $ pointwise and Lebesgue dominated convergence theorem, it is not difficult to show now that
$\mathbb{E}\left[ \left( \mathcal{Z}(\epsilon)-\mathcal{Z}\right)^2\right]$
 converges to zero, i.e. the strong convergence in $L^2(\Omega)$.
Finally by an immediate application of Lemma \ref{lem LA31}, the term $A_{31}(\epsilon,t,\eta)$ expressed in \eqref{eq FormA31} converges to $\frac{\sigma^2}{2}\mathbb{E}[\mathcal{Z}]$
which equals $A_{31}(t,\eta)$.
\\
$\bullet$ The term $A_{32}(\epsilon,t, \eta)$ converges to zero. In fact using $\1_{]t-T+\epsilon,0]^2}\leq \1_{[t-T,0]^{2}} $ and then the Cauchy-Schwarz
inequality we obtain
{\footnotesize{
\[
\begin{split}
&\mathbb{E}\left[
\langle D^{2}H \left( \alpha Y_{T}^{t+\epsilon,\eta}+(1-\alpha) Y_{T}^{t+\epsilon,Y_{t+\epsilon}^{t,\eta } }  \right) -D^{2}H\left( Y_{T}^{t,\eta}\right), \1_{]t-T+\epsilon,0]^{2}}\rangle  \;\cdot\;
\frac{(W_{t+\epsilon}-W_{t})^{2}}{\epsilon}
\right]\\
&\leq
\mathbb{E}\left[
\langle D^{2}H \left( \alpha Y_{T}^{t+\epsilon,\eta}+(1-\alpha) Y_{T}^{t+\epsilon,Y_{t+\epsilon}^{t,\eta } }  \right) -D^{2}H\left( Y_{T}^{t,\eta}\right), \1_{[t-T,0]^{2}}\rangle  \;\cdot\;
\frac{(W_{t+\epsilon}-W_{t})^{2}}{\epsilon}
\right]\\
&\leq
\sqrt{
\mathbb{E}\left[ \left| \langle D^{2}H \left( \alpha Y_{T}^{t+\epsilon,\eta}+(1-\alpha) Y_{T}^{t+\epsilon,Y_{t+\epsilon}^{t,\eta } }  \right) -D^{2}H\left( Y_{T}^{t,\eta}\right), \1_{[t-T,0]^{2}}\rangle\right|^{2} \right] } \, \cdot\,\sqrt{3}\\
&
\leq
\sqrt{
\mathbb{E}\left[ \left\| D^{2}H \left( \alpha Y_{T}^{t+\epsilon,\eta}+(1-\alpha) Y_{T}^{t+\epsilon,Y_{t+\epsilon}^{t,\eta } }  \right) -D^{2}H\left( Y_{T}^{t,\eta}\right)\right\|^2_{\left( C([-T,0])\hat{\otimes}^2_\pi\right)^*}
\cdot \left\| \1_{[t-T,0]^{2}}\right\|^{2} \right].}
\end{split}
\]
}}
The latter term converges to zero because $D^{2}H\in C^{0}\left(C([-T,0])\right)$ and $D^{2}H$ has polynomial growth as
we have seen in Remark \ref{rm PMLOP} item 1.\\
$\bullet$ We show that $A_{33}(\epsilon,t,\eta)$ converges to zero.
We rewrite $A_{33}(\epsilon,t,\eta)$ as $\sigma(A_{332}(\epsilon,t,\eta)- A_{331}(\epsilon,t,\eta))$, where
{\footnotesize{
\[
\begin{split}
A_{331}(\epsilon,t,\eta)
&
=
\mathbb{E}\left[
\langle D^2H\left( Y_{T}^{t,\eta}\right),  \frac{\eta(y+T-t+\epsilon)-\eta(y+T-t)}{\epsilon}  \1_{]t-T+\epsilon,0]}(x)\otimes \1_{[-T,t-T]}(y)\rangle\cdot\right.\\
&\left. \hspace{8cm}\cdot
 (W_{t+\epsilon}-W_{t})
\right] \ ,
\\
A_{332}(\epsilon,t,\eta)
&
= \int_{0}^{1}
\mathbb{E}\left[ \langle
D^2 H\left( \alpha Y_{T}^{t+\epsilon,\eta}+(1-\alpha) Y_{T}^{t+\epsilon,Y_{t+\epsilon}^{t,\eta } }  \right),  \right.\\
&
\hspace{0.3cm}
\left. \frac{\eta(y+T-t+\epsilon)-\eta(y+T-t)}{\epsilon} \1_{]t-T+\epsilon,0]}(x)\otimes \1_{[-T,t-T]}(y)  \rangle (W_{t+\epsilon}-W_{t}) \right]d\alpha			\; .
\end{split}
\]
}}
We will show that both $A_{331}(\epsilon,t,\eta)$ and $A_{332}(\epsilon,t,\eta)$ converge to zero.
Denoting
\begin{equation}			\label{eq-def ZETA}
\mathcal{Z}:=\langle
D^{2} H\left( Y_{T}^{t,\eta}\right) \; ,\;
\mathcal{Y} \rangle,
\end{equation}
where
$$
 \mathcal{Y}:= \1_{]t-T+\epsilon,0]}(x)\otimes \1_{[-T,t-T]}(y)\left[ \eta(y+T-t+\epsilon)-\eta(y+T-t) \right] ,
$$
we rewrite
\[
A_{331}(\epsilon,t,\eta)
=\frac{1}{\epsilon}
\mathbb{E}\left[
\mathcal{Z} \cdot  (W_{t+\epsilon}-W_{t})
\right].
\]
Using Proposition \ref{propP1112},
and that $H\in C^{3}(C([-T,0]))$, with polynomial growth
 we get that $\mathcal{Z}$ belongs to $\D^{1,2}$ and
 {\scriptsize{
\begin{equation}		\label{eq MLK}
\begin{split}
&D^{m}_{r}\mathcal{Z}
=\sigma \langle D^{3} H(Y^{t,\eta}_T)\ , \1_{]r-T+t,0]} \otimes \mathcal{Y} \rangle\quad +\quad \langle D^2 H(Y^{t,\eta}_T)\ ,\ D^m_r \mathcal{Y}\rangle \\
&= \sigma \langle
D^{3} H\left( Y_{T}^{t,\eta}\right)   , \1_{]r-T+t,0]}(z)\otimes
\1_{]t-T+\epsilon,0]}(x)\otimes  \1_{[-T,t-T]}(y) \left[ \eta(y+T-t+\epsilon)-\eta(y+T-t) \right]
\rangle,
\end{split}
\end{equation}
}}
because $D^m_r \mathcal{Y}$ is zero.
Using \eqref{eq-def ZETA}, Skorohod integral formulation,
 integration by parts on Wiener space \eqref{eq INTBPWiener}, \eqref{eq MLK} and successively
Fubini's theorem
with respect to the variables
$r$ and $z$ and then integrating with respect to $r$, we obtain
{\footnotesize{
\begin{equation}			\label{THG}			
\begin{split}
A_{331}(\epsilon,t,\eta)
&
=\frac{1}{\epsilon}
\mathbb{E}\left[
\mathcal{Z} \cdot  (W_{t+\epsilon}-W_{t})
\right]=\frac{1}{\epsilon}
\mathbb{E}\left[
\mathcal{Z} \cdot  \overline W_{\epsilon}
\right]=
\frac{1}{\epsilon}
\mathbb{E}\left[
\mathcal{Z} \cdot \int_{0}^{\epsilon}\delta \overline W_{u}
\right]=\\
&
\hspace{-1cm}
=
\frac{1}{\epsilon} \mathbb{E}\left[
\int_{0}^{\epsilon}D^{m}_{r} \mathcal{Z} \;dr
\right]
=\\
&\hspace{-1cm}
=
\frac{\sigma}{\epsilon}
\mathbb{E}
\left[  \int_{0}^{\epsilon} \
\langle
D^{3} H\left( Y_{T}^{t,\eta}\right) \; , \right.\\
&\hspace{-0,3cm}
\left. \;
\1_{]r-T+t,0]}(z)\otimes \1_{]t-T+\epsilon,0]}(x)\otimes \1_{[-T,t-T]}(y)\left[ \eta(y+T-t+\epsilon)-\eta(y+T-t) \right]
\rangle \ dr
\right]	\\
&\hspace{-1cm}
=
\frac{\sigma}{\epsilon}
\mathbb{E}
\left[  \
\langle
D^{3} H\left( Y_{T}^{t,\eta}\right) \; , \right.\\
&\hspace{-0,8cm}
\left. \;
 \int_{0}^{\epsilon} \1_{]r-T+t,0]}(z) \ dr \ \otimes \1_{]t-T+\epsilon,0]}(x)\otimes \1_{[-T,t-T]}(y)\left[ \eta(y+T-t+\epsilon)-\eta(y+T-t) \right]
\rangle \
\right].
\end{split}
\end{equation}
}}
Analyzing the term $ \int_{0}^{\epsilon} \1_{]r-T+t,0]}(z) \ dr$ analogously to \eqref{THJ1} and \eqref{eq45GTB} we can establish the convergence of $A_{331}(\epsilon,t,\eta)$. In fact
the third order Fr\'echet derivative of $H$, denoted by $D^{3}H$, is
a map from $C([-T,0])$ to the dual of the triple projective tensor product
of $C([-T,0])$, i.e. $\left( C([-T,0]) \hat{\otimes}_{\pi}^{3}\right)^{\ast}$. We recall that, given a general Banach space $E$ equipped
with its norm $\|\cdot\|_{E}$ and $x, y, z$ three elements of
$E$, then the norm of an elementary element of the tensor product $x\otimes y\otimes z$ which belongs to $E\otimes^{3}$ is $\|x\|_{E}\cdot \|y\|_{E}\cdot\|z\|_{E}$.
We remark that the trilinear form $(\phi,\varphi,\psi)\mapsto \langle  D^3H\left( Y^{t,\eta}_T\right),  \phi\otimes \varphi\otimes \psi\rangle $ extends from $C([-T,0])\times C([-T,0])\times C([-T,0])$ to
$\phi,\varphi,\psi:[-T,0]\longrightarrow \R$ as a Borel bounded map. Indeed the application is a measure in each component.
Consequently
\[
\begin{split}
&
\arrowvert
\langle
D^{3} H\left( Y_{T}^{t,\eta}\right) \; ,\;
\1_{]t-T+\epsilon,0]}(x)\otimes \1_{[-t,0]}(y)\left[ \eta(y+\epsilon)-\eta(y) \right]\otimes  \1_{]r-T+t,0]}(z)
\rangle
\arrowvert
\leq\\
&
\\
&
\leq
\sup_{\|\phi\|_{\infty}\leq 1,\| \varphi\|_{\infty}\leq 1,\|\psi\|_{\infty}\leq 1}
\vert
\langle D^3 H\left( Y^{t,\eta}_T\right), \phi\otimes \varphi\otimes \psi \rangle
\vert\cdot  \varpi_{\eta}(\epsilon) \\
&= \left\|
D^{3} H\left( Y_{T}^{t,\eta}\right)  \right\|_{\left( C([-T,0]) \hat{\otimes}_{\pi}^{3}\right)^{\ast}} \, \cdot \,  \varpi_{\eta}(\epsilon)\xrightarrow[\epsilon\longrightarrow 0]{a.s.}0,
\end{split}
\]
since $\varpi_{\eta}(\epsilon)$ is the modulus of continuity of $\eta$.
By the polynomial growth of $D^{3}H$, \eqref{eq MAGG1}, the fact that for any given Brownian motion
$\bar{W}$, $\sup_{x\leq T}\vert \bar{W_{x}}\vert$ has all moments and finally the
Lebesgue dominated convergence theorem we conclude that \eqref{THG} converges to zero, therefore $A_{331}(\epsilon, t, \eta)$ converges to zero.

At this point we should establish the convergence
to zero of $A_{332}(\epsilon,t,\eta)$. This can be done using, again as above, integration by parts on Wiener space \eqref{eq INTBPWiener}.
However there are several technicalities that we have to omit.
\\
$\bullet$ We show finally that $A_{34}(\epsilon,t,\eta)$ converges to zero.\\
We rewrite term $A_{34}(\epsilon,t,\eta)$ as
{\footnotesize{
\[
\begin{split}
A_{34}(\epsilon,t, \eta)&
= \sigma
\int_{0}^{1}
\mathbb{E}\left[
\langle \left( D^{2}H\left( \alpha Y_{T}^{t+\epsilon,\eta}+(1-\alpha) Y_{T}^{t,\eta}  \right) -D^{2}H\left( Y_{T}^{t,\eta}\right) \right), \right. \\
&\left.
 \frac{\eta(y+T-t-\epsilon)-\eta(0)- \sigma(W_{T}(y)+W_{t+\epsilon})}{\epsilon}\1_{]t-T+\epsilon,0]}(x)\otimes \1_{]t-T,t-T+\epsilon]}(y) \rangle \cdot  \right. \\
 &\left. \hspace{5cm} \cdot  (W_{t}-W_{t+\epsilon})  \right] d\alpha
\end{split}
\]
}}
as $A_{34}(\epsilon, t,\eta)= \sigma(A_{341}(\epsilon, t,\eta)-A_{342}(\epsilon, t,\eta))$
where
{\footnotesize{
\[
\begin{split}
A_{341}(\epsilon, t,\eta)&
=
\int_{0}^{1}
\mathbb{E}\left[
\langle D^{2}H\left( \alpha Y_{T}^{t+\epsilon,\eta}+(1-\alpha) Y_{T}^{t,\eta}   \right), \right. \\
&\left.
 \frac{\eta(y+T-t-\epsilon)-\eta(0)- \sigma(W_{T}(y)+W_{t+\epsilon})}{\epsilon}\1_{]t-T+\epsilon,0]}(x)\otimes \1_{]t-T,t-T+\epsilon]}(y) \rangle \cdot
  \right. \\
 & \left. \hspace{5cm} \cdot  (W_{t}-W_{t+\epsilon})  \right] d\alpha \ , \\
A_{342}(\epsilon, t,\eta)&
=
\int_{0}^{1}
\mathbb{E}\left[
\langle D^{2}H\left( Y_{T}^{t,\eta}\right),  \right. \\
&\left.
 \frac{\eta(y+T-t-\epsilon)-\eta(0)- \sigma(W_{T}(y)+W_{t+\epsilon})}{\epsilon}\1_{]t-T+\epsilon,0]}(x)\otimes \1_{]t-T,t-T+\epsilon]}(y) \rangle \cdot \right. \\
 & \left. \hspace{5cm} \cdot (W_{t}-W_{t+\epsilon})  \right] d\alpha=\\
 &
=\mathbb{E}\left[
\langle D^{2}H\left( Y_{T}^{t,\eta}\right),  \right. \\
&\left.
 \frac{\eta(y+T-t-\epsilon)-\eta(0)- \sigma(W_{T}(y)+W_{t+\epsilon})}{\epsilon}\1_{]t-T+\epsilon,0]}(x)\otimes \1_{]t-T,t-T+\epsilon]}(y) \rangle \cdot \right. \\
 & \left. \hspace{5cm} \cdot  (W_{t}-W_{t+\epsilon})  \right] .\\
\end{split}
\]
}}
Firstly we show that $A_{342}$ converges to zero. It holds in fact
\[
A_{342}(\epsilon, t,\eta)=\frac{1}{\epsilon}\mathbb{E}\left[  \mathcal{Z} \cdot  (W_{t}-W_{t+\epsilon})  \right]  =\frac{1}{\epsilon}\mathbb{E}\left[  \mathcal{Z} \cdot  \overline W_{\epsilon}  \right]  = \frac{1}{\epsilon}\mathbb{E}\left[  \mathcal{Z} \cdot  \int_{0}^\epsilon \delta \overline W_r \right],
\]
where
{\scriptsize{
\be
\mathcal{Z}:=\langle D^{2}H\left( Y_{T}^{t,\eta}\right),  \1_{]t-T+\epsilon,0]}(x)\otimes \left[ \eta(y+T-t-\epsilon)-\eta(0)- \sigma W_{T}(y)+
\sigma W_{t+\epsilon}\right]  \1_{]t-T,t-T+\epsilon]}(y) \rangle .
\ee
}}
Since $D^2H$ has polynomial growth and it is of class $C^1$, by Proposition  \ref{propP1112},
 $\mathcal{Z}\in \D^{1,2}$. Then
  the integration by parts on Wiener space gives
\be		\label{eqERT}
A_{342}(\epsilon, t,\eta)=\frac{1}{\epsilon}\, \mathbb{E} \left[ \int_0^\epsilon D^m_r  \mathcal{Z}\  dr \right] \ .
\ee
According to Proposition \ref{propP1112}, equation \eqref{eqDIFF} for $n=2$ and setting
\be
  \mathcal{Y}:=
\1_{]t-T+\epsilon,0]}(x)\otimes \left[ \eta(y+T-t-\epsilon)-\eta(0)-
\sigma W_{T}(y)+ \sigma W_{t+\epsilon}\right]  \1_{]t-T,t-T+\epsilon]}(y) \ ,
\ee
we get the following expression for the Malliavin derivative of $\mathcal{Z} $ in
the Wiener space associated with $(\bar W_r),$ for $r\in [0, T-t]$:
\be		\label{eqERT1}
D^m_r  \mathcal{Z} = \langle D^3H\left( Y_{T}^{t,\eta}\right), \mathcal{Y}\otimes \1_{]r-T+t, 0]}(z) \rangle + \langle D^2  H\left( Y_{T}^{t,\eta}\right), D^m_r \mathcal{Y}\rangle.
\ee
Replacing \eqref{eqERT1} in \eqref{eqERT} we get that $A_{342}(\epsilon, t,\eta)$ equals a sum of $A_{3421}(\epsilon, t,\eta)$ and $A_{3422}(\epsilon, t,\eta)$ with
\be		\label{eqERT3}
\begin{split}
A_{3421}(\epsilon, t,\eta)&=\frac{1}{\epsilon}\, \mathbb{E} \left[  \int_0^\epsilon  \langle D^3H\left( Y_{T}^{t,\eta}\right), \mathcal{Y}\otimes  \1_{]r-T+t, 0]}(z)  \rangle dr \right] \ ,\\
A_{3422}(\epsilon, t,\eta)&= \frac{1}{\epsilon}\, \mathbb{E} \left[  \int_0^\epsilon  \langle D^2H\left( Y_{T}^{t,\eta}\right),D^m_r \mathcal{Y} \rangle dr \right] \ .
\end{split}
\ee
The term $A_{3421}(\epsilon, t,\eta)$ converges to zero. In fact, similarly to the method developed in detail in \eqref{eq45GTB},
we have
\[
\begin{split}
A_{3421}(\epsilon, t,\eta)&=
\frac{1}{\epsilon}\, \mathbb{E} \left[  \int_0^\epsilon \langle D^3H\left( Y_{T}^{t,\eta}\right), \mathcal{Y}\otimes \1_{]z-T+t, 0]}(r)\rangle \  dr \right]\\
&= \frac{1}{\epsilon}\, \mathbb{E} \left[   \langle D^3H\left( Y_{T}^{t,\eta}\right), \mathcal{Y}\otimes  \int_0^\epsilon \1_{]z-T+t, 0]}(r) \  dr  \rangle \right]
\end{split}
\]
and
\[
\frac{1}{\epsilon}
 \int_0^\epsilon \1_{]z-T+t, 0]}(r)\rangle \  dr\leq \frac{ \epsilon\wedge (z+T-t)}{\epsilon} \xrightarrow [\epsilon\longrightarrow 0]{} 1  \ .
\]
Then by polynomial growth of $D^{3}H$, \eqref{eq MAGG}, the usual property that given any Brownian motion $\bar{W}$, $\sup_{x\leq T}\vert \bar{W_{x}}\vert$ has all moments, the convergence of $\mathcal{Y}$ to zero and
through the application of Lebesgue dominated convergence theorem we conclude that first term in $A_{3421}(\epsilon, t,\eta)$ converges to zero.\\
Concerning the term $A_{3422}(\epsilon, t,\eta)$ we firstly need to compute the  Malliavin derivative of $\mathcal{Y}$:
{\footnotesize{
\be 	\label{eqERT2}
\begin{split}
D^m_r\  \mathcal{Y} &=\1_{]t-T+\epsilon,0]}(x)\otimes  D^m_r \left[ \eta(y+T-t-\epsilon)-\eta(0)- \sigma W_{T}(y)+ \sigma W_{t+\epsilon}\right]  \1_{]t-T,t-T+\epsilon]}(y)\\
&= \sigma \1_{]t-T+\epsilon,0]}(x)\otimes  D^m_r \left[  W_{t+\epsilon}-
W_{T+y}\right]  \1_{]t-T,t-T+\epsilon]}(y) \\
&=\sigma \1_{]t-T+\epsilon,0]}(x)\otimes  D^m_r \left[ \overline W_{\epsilon}-\overline W_{T+y-t}\right]  \1_{]t-T,t-T+\epsilon]}(y)\\
&= \sigma \1_{]t-T+\epsilon,0]}(x)\otimes  \1_{[T+y-t,\epsilon]}(r)\cdot \1_{]t-T,t-T+\epsilon]}(y),
\end{split}
\ee
}}
since by usual property of Malliavin derivative $ D^m_r \left[ \overline W_{\epsilon}-\overline W_{T+y-t}\right]=\1_{[T+y-t,\epsilon]}(r)$.
Now replacing \eqref{eqERT2} in \eqref{eqERT3} we have, similarly to the method developed in detail in \eqref{eq45GTB},
\[
A_{3422}(\epsilon, t,\eta) = \frac{1}{\epsilon}\, \mathbb{E} \left[    \langle D^2H\left( Y_{T}^{t,\eta}\right), \int_0^\epsilon D^m_r \mathcal{Y} \ dr  \rangle \right]
\]
and
\[
\frac{\sigma}{\epsilon}
 \int_0^\epsilon D^m_r \mathcal{Y} \ dr
\leq \sigma \frac{ \epsilon\wedge (T+y-t)}{\epsilon} \xrightarrow [\epsilon\longrightarrow 0]{} 0   \ .
\]
We remark that $T+y-t\in [0,\epsilon]$ since $y\in [t-T, t-T+\epsilon]$. Again by polynomial growth of $D^{2}H$, \eqref{eq MAGG}, by
 the usual property that for the Brownian motion $\bar{W}$, $\sup_{x\leq T}\vert \bar{W_{x}}\vert$ has all moments and applying Lebesgue dominated convergence theorem, we conclude that first term in $A_{3422}(\epsilon, t,\eta)$ converges to zero. Finally $A_{342}(\epsilon, t,\eta)$ converges to zero.

By similar arguments, even though technically a little bit more involved, also $A_{341}(\epsilon, t,\eta)$ converges to zero. This finally proves that $A_{34}(\epsilon, t,\eta) \xrightarrow [\epsilon\longrightarrow 0]{} 0$.\\
$\bullet$ We are now able to express $\partial_{t} u:[0,T]\times C([-T,0])\longrightarrow \R$.
For $t \in [0,T]$, $\partial_{t}u(t,\eta)$ is given by the convergence of term \eqref{eq ZA} to a sum of three terms different form zero:
\[
\partial_{t}u(t,\eta)
= I_{1}(t,\eta)+I_{3}(t,\eta)+A_{31}(t,\eta) \ , \] i.e.
\begin{equation}		\label{eq dtu}
\partial_{t}u(t,\eta)
=-
\mathbb{E}\left[ \int_{]-t,0]} D_{x-T+t}H\left( Y_{T}^{t,\eta} \right)d^- \eta(x) + \frac{\sigma^2}{2}
\langle D^{2}H\left( Y_{T}^{t,\eta}\right) ,\1_{]t-T,0]}\otimes^2 \rangle \right]. \\
\end{equation}
$\bullet$ \textbf{ The path-dependent heat equation.} \\
Taking into account  \eqref{E111},
the second order Fr\'echet derivative \eqref{eq D2u}
and the time derivative \eqref{eq dtu} it finally follows that
$u$ solves \eqref{kolmogorov}.
\end{proof}

%
%
%
%

%
%
%


\section{Appendix: Malliavin and Fr\'echet derivatives}

We need some technical results concerning the link between Fr\'echet and Malliavin derivatives in a separable Banach space, that for the moment ,we set to be equal to $\R$.
We need to apply Malliavin calculus related to the Brownian motion.
Let $T>0$ and $t\in [0,T]$ be fixed.
We recall that
\be		\label{Wbar}
\overline W_x :=W_{t+x}-W_t, \quad x\in [0,T-t]  \ .
\ee
So the Wiener space will be $C([0,T-t])$ with variable parameter in $[0,T-t]$ and based on $\overline W$. We consider the window Brownian element
 $\overline W_{T-t}(\cdot)$ with values in $C([-(T-t),0])$, defined as
\[
\overline W_{T-t}(x)=W_{t+T-t}(x)-W_t=W_{T+x}-W_{t}  \quad x\in [-(T-t),0].
\]
\begin{lem}		\label{lem L11}
Let $G:C([-(T-t),0])\longrightarrow \R$ of class $C^1$ such that $DG$ has polynomial growth. Let $\mathcal{Y}\in \mathbb{D}^\infty$.
Then
\be
G\left(\sigma \overline W_{T-t}(\cdot) \right)
\ee
belongs to $\mathbb{D}^{1,2}$ and
\be		\label{eqEL1}
\begin{split}
D^m_{r}\left( G( \sigma \overline W_{T-t}(\cdot)) \  \mathcal{Y}  \right)=&
 \sigma \int_{]r-(T-t),0]} \left( D_{dy} G \right) \left(\sigma \overline W_{T-t}(\cdot) \right) \mathcal{Y}  \\
 &+
G\left( \sigma \overline W_{T-t}(\cdot) \right) D^m_r \mathcal{Y} \quad \qquad \quad r\in [0, T-t]\;  a.e.
\end{split}
\ee
\end{lem}
\begin{proof}
The proof of this result needs some boring technicalities involving the approximation of a continuous function and its polygonal approximation. Formula \eqref{eqEL1} is stated in a particular case
for instance in \cite{nualartSEd}, Example 1.2.1.
\end{proof}
A consequence of previous lemma is the possibility of the differentiating
\be		\label{eqEL11}
h=F\left( Y^{t,\eta}_T \right), \quad F:C([-T,0])\longrightarrow \R
\ee
of class $C^1$ Fr\'echet. We remark that
\be
Y^{t,\eta}_{T}=G_{\eta}\left(\sigma \overline W_{T-t}(\cdot) \right),
\ee
where
\be
G_{\eta}:C[-(T-t),0]\longrightarrow C([-T,0])
\ee
given by
\be
G_{\eta}(\gamma)=
\left\{
\begin{array}{ll}
\eta(x+T-t) & x\in  [-T,t-T[\\
\eta(0)+\gamma(T-t+x) & x\in [t-T,0].
\end{array}
\right.
\ee
By Lemma \ref{lem L11}, if $\mathcal{Y}\in \mathbb{D}^{\infty}$,
\be		\label{eqEL111}
\begin{split}
D^m_{x}\left( h \mathcal{Y} \right)=& \sigma \int_{]x-T+t,0]}D_{dy} \left( F \circ G_{\eta} \right) \left(\sigma \overline W_{T-t}(\cdot) \right) \mathcal{Y}\\
& +
F\circ G_{\eta} \left( \sigma \overline W_{T-t}(\cdot) \right) \; D^m_x\mathcal{Y}\quad\qquad \quad \quad x\in [0,T-t].
\end{split}
\ee
\begin{rem}		\label{remR1}
We remark that $\forall \ \gamma \in C([-T+t,0])$
\[
D\left( F\circ G_\eta \right)\in \mathcal{M}([-T+t,0]).
\]
\end{rem}
We have, for $\zeta \in C([-T+t,0])$,
\be		
\int D_{dy}\left(F\circ G_\eta \right)(\gamma)\ \zeta(y)=\int_{[t-T,0]}D_{dy}F\left(  G_\eta(\gamma) \right)\ \zeta(y).
\ee
So \eqref{eqEL111} gives, for $x\in [0,T-t]$,
\be		\label{eqEL1111}
\begin{split}
D^m_x(h\mathcal{Y})&= \sigma \int_{]x-T+t,0]}\left( D_{dy} F\right) \left( G_\eta(\sigma \overline W_{T-t}) \right)\mathcal{Y}\quad +\quad
\left( F\circ G_\eta\right) \left(\sigma \overline W_{T-t} (\cdot) \right) \ D^m_x\mathcal{Y}\\
&=\sigma \int_{]x-T+t,0]}\left( D_{dy}F\right)\left(Y^{t,\eta}_T \right)\ \mathcal{Y} \quad +\quad F\left( Y^{t,\eta}_T\right) D^m_x\mathcal{Y}.
\end{split}
\ee
At this point we have proved the following.

\begin{prop}		\label{propP11}
 Let
$H:C([-T,0])\longrightarrow \R$ of class $C^1$-Fr\'echet with polynomial
growth. Let $\mathcal{Y} \in {\mathbb D}^\infty.$ Then
$  H(Y^{t,\eta}_T)  \mathcal{Y}$ belongs to ${\mathbb D}^{1,2}$ and
\be		\label{propP111}
D^m_r\left( H(Y^{t,\eta}_T) \ \mathcal{Y} \right)=\sigma \int_{]x-T+t,0]} \left(   D_{dy} H\right)(Y^{t,\eta}_T)\ \mathcal{Y} \quad +\quad F\left(Y^{t,\eta}_T \right) D_r^m \mathcal{Y} \ .
\ee
\end{prop}

Previous proposition admits a generalization to the case when $H:C([-T,0])\longrightarrow \R$ is replaced by a functional
\be
C([-T,0])\longrightarrow   \underbrace{\left( C([-T,0])\hat\otimes_\pi \cdots \hat\otimes_{\pi} C([-T,0])\right)} _{n \mbox{ times}} 		\quad\quad n\geq 1.
\ee
Typically an example will be $D^n H$.
We recall that
\[
\Big(  \underbrace{ C([-T,0])\hat\otimes_\pi \cdots\hat\otimes_{\pi} C([-T,0]) } _{n \mbox{ times}} \Big)^*
\]
can be isomorphically identified with the space of $n$-multilinear continuous
 functionals on $ C([-T,0])$. 
 Proposition \ref{propP11} can be generalized as follows.
\begin{prop}	\label{propP1112}
Let $H:C([-T,0])\longrightarrow \R$ of class $C^{n+1}$-Fr\' echet
such that $D^{n+1} F$ has polynomial growth.
Let $\mathcal{Y}\in \mathbb{D}^{\infty}\Big( \underbrace{  C([0,T-t])\hat\otimes_\pi \cdots \hat\otimes_{\pi} C([0,T-t]) } _{n \mbox{ times}} \Big)$.
Then $\langle D^{n} H(Y^{t,\eta}_T), \shy \rangle$ belongs to ${\mathbb D}^{1,2}$.
Moreover, for a.e. $r\in [0,T-t]$ we have
\be \label{eqDIFF}
\begin{split}
D^m_r\left( \langle D^n H(Y^{t,\eta}_T)\ ,\ \mathcal{Y} \rangle \right)=&
\sigma \langle D^{n+1} H(Y^{t,\eta}_T)\ , 1_{]r-T+t,0]} \otimes \mathcal{Y} \rangle\\
& + \langle D^nH(Y^{t,\eta}_T)\ ,\ D^m_r \mathcal{Y}\rangle \ .
\end{split}
\ee
\end{prop}

\begin{rem}	\label{remR3}
The function $1_{]r-T+t,0]}$ can be considered as a test function
$\zeta_0$. Indeed for fixed $\zeta_1,\cdots, \zeta_n\ \in C([-T,0])$,
\be
\zeta_0 \mapsto D^{n+1} H(Y^{t,\eta}_T)\left( \zeta_0\otimes \zeta_1\otimes \cdots \zeta_n \right)
\ee
is a measure.
\end{rem}
\begin{proof} Avoiding to state too abstract results, the proof of Proposition \ref{propP1112} is based on a generalization of Lemma \ref{lem L11} replacing the value space $\R$ with the
separable Banach space $B$,
setting $B=C([-T,0])\hat\otimes_\pi \cdots \hat\otimes_{\pi} C([-T,0])$.
\end{proof}

\begin{lem}		\label{lem L12}
Let $B$ be a separable Banach space. Let
\[
G:C([-T+t,0])\longrightarrow B^*
\]
of class $C^1$ Fr\'echet with polynomial growth. Let $\mathcal{Y}\in \mathbb{D}^\infty(B)$.

Then $
G\left(\overline W_{T-t}(\cdot) \right)\ \mathcal{Y}\ \in \mathbb{D}^{1,2}(B)$
 and
 \be
 \begin{split}
& D_r \left( _{B^*}\langle G\left( \overline W_{T-t}(\cdot) \right) \ ,\ \mathcal{Y} \rangle_B \right)\\
&\hspace{2cm}=
_{(C([-T,0])\hat\otimes_\pi B)^*} \langle DG\left( \overline W_{T-t}(\cdot) \right), 1_{]r-T+t,0]}\otimes \mathcal{Y} \rangle_{C([-T,0])\hat\otimes_\pi B} \\
& \hspace{4cm}+
\langle G\left(\overline W_{T-t}(\cdot) \right)\ , \ D^m_r\mathcal{Y}\rangle
\end{split}
 \ee
\end{lem}
\begin{rem}
\begin{enumerate}
\item $DG:C([-T+t, 0])\longrightarrow (C([-T,0])\hat\otimes_\pi B)^*$.
\item Proposition \ref{propP1112} will be used for $n=1,2,3$.
\item $D^m_r\mathcal{Y}\in B$ for almost all $r$.
\end{enumerate}
\end{rem}

\section*{Acknowledgement}
The authors wish to thank the Referee
 for the careful reading of the paper and
the Associated Editor and the Editor in Chief as well.
The second named author would like to thank the first one
for the kind invitation in the framework of two GNAMPA
projects financially supported by Istituto di Alta Matematica
Francesco Severi.


\begin{thebibliography}{99}
\bibitem{cartan}
Henri Cartan.
\newblock {\em Calcul diff\'erentiel}.
\newblock Hermann, Paris, 1967.

\bibitem{rosestolato_cosso}
Andrea Cosso, Salvatore Federico, Fausto Gozzi, Mauro Rosestolato, and Nizar
  Touzi.
\newblock Path-dependent equations and viscosity solutions in infinite
  dimension.
\newblock {\em Ann. Probab.}, 46(1):126--174, 2018.

\bibitem{cosso_russo15b}
Andrea Cosso and Francesco Russo.
\newblock Strong-viscosity solutions: classical and path-dependent {PDE}s.
\newblock {\em Osaka J. Math.}, 56(2):323--373, 2019.

\bibitem{cosso_russo14}
Andrea Cosso and Francesco Russo.
\newblock Functional and {B}anach space stochastic calculi: path-dependent
  {K}olmogorov equations associated with the frame of a {B}rownian motion.
\newblock In {\em Stochastics of environmental and financial
  economics---{C}entre of {A}dvanced {S}tudy, {O}slo, {N}orway, 2014--2015},
  volume 138 of {\em Springer Proc. Math. Stat.}, pages 27--80. Springer, Cham,
  2016.

\bibitem{cosso_russo15a}
Andrea Cosso and Francesco Russo.
\newblock Functional {I}t\^o versus {B}anach space stochastic calculus and
  strict solutions of semilinear path-dependent equations.
\newblock {\em Infin. Dimens. Anal. Quantum Probab. Relat. Top.},
  19(4):1650024, 44, 2016.

\bibitem{DGR}
Cristina Di~Girolami and Francesco Russo.
\newblock Infinite dimensional stochastic calculus via regularization and
  applications.
\newblock {\em Preprint HAL-INRIA},
  http://hal.archives-ouvertes.fr/inria-00473947/fr/(Unpublished), 2010.

\bibitem{DGRnote}
Cristina Di~Girolami and Francesco Russo.
\newblock Clark-{O}cone type formula for non-semimartingales with finite
  quadratic variation.
\newblock {\em C. R. Math. Acad. Sci. Paris}, 349(3-4):209--214, 2011.

\bibitem{DGR1}
Cristina Di~Girolami and Francesco Russo.
\newblock {G}eneralized covariation for {B}anach space valued processes,
  {I}t\^o formula and applications.
\newblock {\em Osaka Journal of Mathematics}, 51:729--783, 2014.

\bibitem{dupire}
Bruno Dupire.
\newblock Functional {I}t\^o calculus.
\newblock {\em Bloomberg Portfolio Research, Paper No. 2009-04-FRONTIERS},
  http://ssrn.com/abstract=1435551, 2009.

\bibitem{zanco}
Franco Flandoli and Giovanni Zanco.
\newblock An infinite-dimensional approach to path-dependent {K}olmogorov
  equations.
\newblock {\em Ann. Probab.}, 44(4):2643--2693, 2016.

\bibitem{follWuYor}
Hans F{\"o}llmer, Ching-Tang Wu, and Marc Yor.
\newblock On weak {B}rownian motions of arbitrary order.
\newblock {\em Ann. Inst. H. Poincar\'e Probab. Statist.}, 36(4):447--487,
  2000.

\bibitem{nevpg}
Jacques Neveu.
\newblock {\em Processus al\'eatoires gaussiens}.
\newblock S\'eminaire de Math\'ematiques Sup\'erieures, No. 34 (\'Et\'e, 1968).
  Les Presses de l'Universit\'e de Montr\'eal, Montreal, Que., 1968.

\bibitem{nualartSEd}
David Nualart.
\newblock {\em The {M}alliavin calculus and related topics}.
\newblock Probability and its Applications (New York). Springer-Verlag, Berlin,
  second edition, 2006.

\bibitem{pengBSDEs}
Shige Peng and Falei Wang.
\newblock B{SDE}, path-dependent {PDE} and nonlinear {F}eynman-{K}ac formula.
\newblock {\em Sci. China Math.}, 59(1):19--36, 2016.

\bibitem{rv}
Francesco Russo and Pierre Vallois.
\newblock Int\'egrales progressive, r\'etrograde et sym\'etrique de processus
  non adapt\'es.
\newblock {\em C. R. Acad. Sci. Paris S\'er. I Math.}, 312(8):615--618, 1991.

\bibitem{rv1}
Francesco Russo and Pierre Vallois.
\newblock Forward, backward and symmetric stochastic integration.
\newblock {\em Probab. Theory Related Fields}, 97(3):403--421, 1993.

\bibitem{Rus05}
Francesco Russo and Pierre Vallois.
\newblock Elements of stochastic calculus via regularization.
\newblock In {\em S\'eminaire de {P}robabilit\'es {XL}}, volume 1899 of {\em
  Lecture Notes in Math.}, pages 147--185. Springer, Berlin, 2007.

\bibitem{rr}
Raymond~A. Ryan.
\newblock {\em Introduction to tensor products of {B}anach spaces}.
\newblock Springer Monographs in Mathematics. Springer-Verlag London Ltd.,
  London, 2002.

\end{thebibliography}
\end{document}